\documentclass{amsart}
\usepackage{graphicx} 
\usepackage{amsfonts, amsmath, amssymb, amsthm}
\usepackage{xcolor}
\usepackage{enumerate}
\usepackage{hyperref}

\newtheorem{theorem}{Theorem}[section]
\newtheorem{lemma}[theorem]{Lemma}
\newtheorem{prop}[theorem]{Proposition}
\newtheorem{coro}[theorem]{Corollary}

\newtheorem{remark}[theorem]{Remark}
\newtheorem{question}[theorem]{Question}

\title[Relatively free algebras of Lie nilpotent associative algebras]{Identities of relatively free algebras of Lie nilpotent associative algebras}
\author{Elitza Hristova}
\address{Institute of Mathematics and Informatics,
Bulgarian Academy of Sciences,
Acad. G. Bonchev Str., Block 8,
1113 Sofia, Bulgaria}
\email{e.hristova@math.bas.bg}
\author{Thiago Castilho de Mello}
\address{Instituto de Ciência e Tecnologia, Universidade Federal de São Paulo, Av. Cesare Monsuetto Giulio Lattes, 1201, S\~ao Jos\'e dos Campos, SP, Brazil}
\email{tcmello@unifesp.br}
\date{\today}


\begin{document}

\begin{abstract}
    In this paper, we consider the relatively free algebra of rank $n$, $F_n(\mathfrak{N}_p)$, in the variety of Lie nilpotent associative algebras of index $p$, denoted by $\mathfrak{N}_p$, over a field of characteristic zero. We describe an explicit minimal basis for the polynomial identities of $F_n(\mathfrak{N}_p)$ when $p=3$ and $p=4$, for all $n$, except for $F_3(\mathfrak{N}_4)$.
    In the general case, we exhibit a lower and an upper bound for the minimal $k$ such that $[x_1,x_2]\cdots[x_{2k-1},x_{2k}]$ is an identity for $F_n(\mathfrak{N}_p)$ for all $n$ and for all $p$.
\end{abstract}

\maketitle

\section{Introduction}
Let $K$ be a field of characteristic zero. In this paper, the word algebra means an associative unitary $K$-algebra. Let $X = \{x_1,x_2, \dots\}$ be a countable set. We denote by $K\langle X \rangle$ the free $K$-algebra, freely generated by $X$. The elements of $K\langle X \rangle$ are polynomials in the non-commuting variables of the set $X$. If $A$ is a $K$-algebra, we say that a polynomial $f = f(x_1,\dots,x_m)\in K\langle X\rangle$ is an identity for $A$ if $f(a_1, \dots, a_m) = 0$, for any $a_i\in A$. If $A$ satisfies a polynomial identity, then $A$ is said to be an {\it algebra with polynomial identity}, or simply a {\it PI-algebra}. We refer the reader to
\cite{D1999, GZ2005} for the basics of PI-algebras. 

An important type of associative algebras are the so-called \emph{Lie nilpotent} associative algebras. They are defined as follows: An algebra $A$ is Lie nilpotent of index at most $p$ if it satisfies the identity $[x_1,x_2,\dots,x_{p+1}] = 0$. Here, $[x_1,x_2,\dots, x_{p+1}]$ is defined inductively by $[x_1,x_2] = x_1x_2-x_2x_1$ (the usual commutator of $x_1$ and $x_2$) and if $p\geq 2$, $[x_1,\dots, x_{p+1}] = [[x_1,\dots,x_{p}],x_{p+1}]$. The class of all algebras which satisfy the identity $[x_1,x_2,\dots,x_{p+1}] = 0$ is called the variety of Lie nilpotent associative algebras of index $p$. Following the tradition, we denote it by $\mathfrak{N}_p$.
For example, the variety $\mathfrak{N}_1$ is the variety of commutative algebras, while the variety $\mathfrak{N}_2$ is the variety of algebras satisfying the identities of the infinite dimensional Grassmann algebra. 

We denote by $I_p$ the two-sided associative ideal in $K\langle X \rangle$ generated by all commutators of length $p$. In the language of PI-algebras, $I_p$ is the {\it T-ideal} generated by $[x_1, \dots, x_p]$ (see Section \ref{sec_Prelim} for the definition of T-ideal and all other necessary notions from the theory of PI-algebras).

The varieties $\mathfrak{N}_p$ and the associated ideals $I_p$ have been studied by many authors. We refer the reader to the papers \cite{L1965, Volichenko_preprint,P2019, Jennings1942,Stoyanova1982(M5), H2023, Gordienko2007, DK2017, DK2019} in which different aspects of these varieties were considered. An interesting aspect, which will be quite useful in this paper, is given by the following theorem due to \cite{L1965, GL1983}.
\begin{theorem} \label{thm_Latyshev} \cite{L1965, GL1983}
    For any $p_1, p_2$, $I_{p_1}I_{p_2}\subset I_{p_1+p_2-2}$.
\end{theorem}
The above can be improved if at least one of $p_1$ or $p_2$ is odd \cite{BJ2013, DK2019}:
\begin{theorem} \label{thm_Bapat_Jordan} \cite{BJ2013, DK2019}
    If at least one of $p_1$ or $p_2$ is odd, $I_{p_1}I_{p_2}\subset I_{p_1+p_2-1}$.
\end{theorem}
Theorem \ref{thm_Bapat_Jordan} is clearly not true if both $p_1$ and $p_2$ are even, since $[x_1,x_2][x_3,x_4]$ is not an identity for the infinite dimensional Grassmann algebra, which satisfies $[x_1,x_2,x_3] = 0$.

Although Theorem \ref{thm_Bapat_Jordan} is not true in general, if we restrict ourselves to polynomials of at most three variables, it is true for any $p_1$ and $p_2$, as proved by Pchelintsev in \cite{P2019}:
\begin{theorem} \label{thm_Pchelintsev} \cite{P2019}
    For any $p_1$ and $p_2$, $I_{p_1}I_{p_2}\cap K\langle x_1,x_2,x_3\rangle \subset I_{p_1+p_2-1}\cap K\langle x_1,x_2,x_3\rangle$.
\end{theorem}

Notice that, translated in the language of polynomial identities, Theorem \ref{thm_Latyshev} means that $[x_1,\dots,x_{p_1}][x_{p_1+1}, \dots, x_{p_1+p_2}]$ 
is a consequence of $[x_1,\dots, x_{p_1+p_2-2}]$, 
or equivalently, that it is a polynomial identity for the relatively free algebra $F(\mathfrak{N}_{p_1+p_2-3}) = \frac{K\langle X \rangle}{I_{p_1 + p_2 - 2}}$. 

On the other hand, the above counterexample and Theorem \ref{thm_Pchelintsev} show us that $[x_1,x_2][x_3,x_4]$ is not a polynomial identity for $F(\mathfrak{N}_2)$, but is a polynomial identity for the relatively free algebra of rank $3$, $F_3(\mathfrak{N}_2) = \frac{K\langle x_1, x_2, x_3\rangle}{K\langle x_1, x_2, x_3\rangle\cap I_3} $.

Thus, a natural question arises in this situation.

For a given $p$, and a given $n$, what are the polynomial identities of $F_n(\mathfrak{N}_p)$?

The structure and identities of relatively free algebras for certain varieties have been explored in earlier works, see e.g., \cite{Procesi1973, Berele1993, KoshlukovMello2013a, KoshlukovMello2013b, GdeM2015,MelloYasumura2021}.

It is clear that if $\mathfrak{V}$ is a variety of algebras, the algebra $F(\mathfrak{V})$ satisfies exactly the same identities as $\mathfrak{V}$, but for relatively free algebras of finite rank $n$, we can only guarantee $\mathrm{Id}(F_n(\mathfrak{V}))\supseteq \mathrm{Id}(\mathfrak{V})$. If there exists some $n$ such that equality holds, we say that $\mathfrak{V}$ has a finite basic rank. Otherwise, $\mathfrak{V}$ has infinite basic rank.
The varieties of infinite basic rank are the ideal setting to study this kind of problems. In such varieties, we have an infinite number of proper inclusions in the following chain of T-ideals
\[K\langle X \rangle \supseteq \mathrm{Id}(F_1(\mathfrak{V}))\supseteq \cdots \supseteq \cdots \mathrm{Id}(F_n(\mathfrak{V}))\supseteq \mathrm{Id}(F_{n+1}(\mathfrak{V})) \supseteq\cdots \]

In this paper, we study identities of the relatively free algebras of finite rank of varieties of Lie nilpotent algebras. The case of $\mathfrak{N}_2$ is well known, since it is the variety generated by the infinite dimensional Grassmann algebra $E$. The identities of $F_{2n}(\mathfrak{N}_2)$ and of $F_{2n+1}(\mathfrak{N}_2)$ are given by 
\[
\mathrm{Id}(F_{2n}(\mathfrak{N}_2)) = \mathrm{Id}(F_{2n+1}(\mathfrak{N}_2)) = \langle [x_1,x_2,x_3], [x_1,x_2]\cdots [x_{2n+1},x_{2n+2}]\rangle^T.
\]
Also, it is interesting to note that $\mathrm{Id}(F_{k} (\mathfrak{N}_2)) = \mathrm{Id}(E_{k})$, where $E_k$ denotes the Grassmann algebra of a $k$-dimensional vector space.

For $\mathfrak{N}_3$, we describe a finite set of generators for $\mathrm{Id}(F_n(\mathfrak{N}_3))$ for every $n$. In addition, for each $n$, we determine the minimal degree of a standard identity satisfied by $F_n(\mathfrak{N}_3)$. 
For $\mathfrak{N}_4$, we describe a minimal set of generators for $\mathrm{Id}(F_n(\mathfrak{N}_4))$ for every $n\neq 3$.

The varieties $\mathfrak{N}_p$ are non-matrix varieties. This means that $M_2(K)\not \in \mathfrak{N}_p$. A consequence of this fact is that every finitely generated algebra in $\mathfrak{N}_p$ satisfies an identity of the type $[x_1,x_2]\cdots[x_{2k-1},x_{2k}]$ for some $k$ (see \cite{Cekanu1980,MischenkoPetrogradskyRegev2011}). 
Another way to derive this fact is from a result of Jennings \cite{Jennings1942}, where the author shows that if $A$ is a finitely generated algebra which is Lie nilpotent, then the ideal $C(A)$, generated by commutators, is nilpotent. In particular, since $F_n(\mathfrak{N}_p)$ is Lie nilpotent and finitely generated,  $C(F_n(\mathfrak{N}_p))$ is a nilpotent ideal, which means that it satisfies an identity of the type $[x_1,x_2]\cdots[x_{2k-1},x_{2k}]$ for some $k$.
In this paper, we find for any $p$ and for any $n$, a lower and an upper bound for  the minimal $k$ such that $[x_1,x_2]\cdots[x_{2k-1},x_{2k}]$ is an identity for $F_n(\mathfrak{N}_p)$. 

In the last section of the paper, we briefly discuss the asymptotic equivalence of the varieties $\mathfrak{N}_p$.

\section{Preliminaries} \label{sec_Prelim}

In this section, we recall the basic definitions and properties of PI-algebras and commutator ideals, which we will need in the sequel. As before, $K \left\langle X \right \rangle$ denotes the free associative algebra generated by the countable set $X = \{x_1, x_2, \dots, x_n, \dots\}$. A {\it T-ideal} in $K \left\langle X \right \rangle$ is any ideal that is closed under all $K$-algebra endomorphisms of $K \left\langle X \right \rangle$. If $A$ is a PI-algebra, then the set of polynomial identities of $A$ is a $T$-ideal in $K \left\langle X \right \rangle$, which we denote by $\mathrm{Id}(A)$. One example is given by the ideal $I_p$, which, as already mentioned in the Introduction, is the $T$-ideal generated by the commutator $[x_1, \dots, x_p]$. For a given set $S\subseteq K\langle X \rangle$, we denote by $\langle S\rangle ^T$ the smallest $T$-ideal of $K\langle X \rangle$ containing $S$. If $S=\{f_1, \dots, f_n\}$, we denote it simply by $\langle f_1, \dots, f_n\rangle^T$.

The class of all algebras that satisfy a certain set of polynomial identities $\mathcal{F}$ is called the variety of algebras defined by $\mathcal{F}$ and is denoted by $\mathrm{var}(\mathcal{F})$. Conversely, if $\mathfrak{V}$ is a variety of PI-algebras, then by $\mathrm{Id}(\mathfrak{V})$ we denote the set of identities satisfied by all algebras in $\mathfrak{V}$, which is naturally a $T$-ideal. 

Let $\mathfrak{V}$ be a variety of PI-algebras and let $I = \mathrm{Id}(\mathfrak{V})$ be the T-ideal of identities of $\mathfrak{V}$. The quotient algebra $F(\mathfrak{V}) = \frac{K \left\langle X \right \rangle}{I}$ is called the {\it relatively free algebra} in the variety $\mathfrak{V}$. One can also define the relatively free algebra of finite rank as follows. If $n$ is a positive integer, then the relatively free algebra of rank $n$ in the variety $\mathfrak{V}$ is defined as
\[
F_n(\mathfrak{V}) = \frac{K \left\langle x_1, \dots, x_n \right \rangle} {I \cap K \left\langle x_1, \dots, x_n \right \rangle}.
\]
In particular, $F_n(\mathfrak{N}_p) =\frac{K \left\langle x_1, \dots, x_n \right \rangle}  {I_{p+1}\cap  K \left\langle x_1, \dots, x_n \right \rangle}$ is the relatively free algebra of rank $n$ in the variety $\mathfrak{N}_p$. 

We also define by $L_p$ the Lie ideal in $K \left\langle X \right \rangle$ generated by all commutators of length $p$. Then, it holds that $I_p = K \left\langle X \right \rangle \cdot L_p$. 

The following theorem summarizes several known properties of the commutator ideals $L_p$ and $I_p$ which will be very useful in what follows. We are only interested in fields of characteristic zero and that is why we state all results for $K$ being of characteristic zero. For more general fields see \cite{DK2019}. Statements (ii), (iv), and (vi) from the theorem below, appear also as Theorems \ref{thm_Latyshev}, \ref{thm_Bapat_Jordan}, and \ref{thm_Pchelintsev} in the Introduction. 

\begin{theorem} \label{thm_products_of_commutators} Let $p_1$ and $p_2$ be positive integers greater than $1$. The following properties hold:
\begin{itemize}
\item [(i)] $[L_{p_1}, L_{p_2}] \subset L_{p_1 + p_2}$;
\item [(ii)] (\cite{L1965}, \cite{GL1983})
\[
I_{p_1} I_{p_2} \subset I_{p_1 + p_2 - 2};
\]
\item [(iii)] (\cite{BJ2013}) If $p_2$ is odd then
\[
[L_{p_1}, I_{p_2}] \subset L_{p_1 + p_2}
\]
\item [(iv)] (\cite{BJ2013}, \cite{DK2019}) Whenever $p_1$ or $p_2$ is odd then
\[
I_{p_1} I_{p_2} \subset I_{p_1 + p_2 - 1};
\]
\item [(v)] (\cite{H2023}) Let $p_1$ and $p_2$ be even and let $c_1 \in L_{p_1-1}$ and $c_2 \in L_{p_2 -1}$. Then, for any $z \in K\left\langle X \right\rangle$
\[
[c_1, z] [c_2, z] \in I_{p_1 + p_2 -1}.
\]
\item [(vi)] (\cite{P2019}) Consider the free associative algebra $K\left\langle x_1, x_2, x_3 \right\rangle$ generated by three elements. Then, for any $p_1$ and $p_2$
\[
I_{p_1} I_{p_2} \cap K\left\langle x_1, x_2, x_3 \right\rangle \subset I_{p_1 + p_2 - 1} \cap K\left\langle x_1, x_2, x_3 \right\rangle.
\]
\end{itemize}
\end{theorem}

The following lemma is an immediate consequence of Theorem \ref{thm_products_of_commutators} (v).

\begin{lemma} \label{lemma_even_comms}
Let $p_1$ and $p_2$ be even integers and let $c_1 \in L_{p_1-1}$ and $c_2 \in L_{p_2 -1}$. Furthermore, let $x,y \in K\left\langle X \right\rangle$. Then the following holds:
\[
[c_1, x][c_2, y] + [c_1, y][c_2, x] \in I_{p_1 + p_2 -1}.
\]
\end{lemma}

\begin{proof}
We take $[c_1, x+y][c_2, x+y]$ and use Theorem \ref{thm_products_of_commutators} (v).
\end{proof}

We will also need the following auxiliary results.

\begin{lemma}\label{CommutatorOfMonomials}
In any associative algebra $A$, if $k, l, m \geq 1$, and $x_1,\dots,x_k$, $y_1,\dots,y_l$, and $z_1, \dots, z_m \in A$, we have
\begin{itemize}
\item[(i)]
    \[[x_1\cdots x_k, y_1\cdots y_l] = \sum_{i=1}^k\sum_{j=1}^l x_1\cdots x_{i-1}y_1\cdots y_{j-1}[x_i,x_j]y_{j+1}\cdots y_lx_{i+1}\cdots x_k.\]
\item[(ii)]
\begin{align*} 
 &[x_1\cdots x_k, y_1 \cdots y_l, z_1 \cdots z_m] = \sum_{i =1}^k\sum_{j = 1}^l \sum_{r=1}^m\\
 & x_1\cdots x_{i-1} y_1 \cdots y_{j-1}z_1 \cdots z_{r-1}[x_i, y_j, z_r]z_{r+1}\cdots z_m y_{j+1}\cdots y_l x_{i+1}\cdots x_k + f,
\end{align*}
where $f \in I_2I_2$.
\end{itemize}
\end{lemma}

\begin{proof}
The proof of Part (i) follows immediately from an induction argument based on the well-known fact that for any $a,b,c\in A$
    \[[ab,c] = a[b,c]+[a,c]b.\]
Then, Part (ii) follows from Part(i).    
\end{proof}

Important objects in the theory of polynomial identities are the so-called proper polynomials. They are defined as follows. We say that a commutator $[x_{i_1}, \dots, x_{i_r}] \in K\langle X \rangle$ is {\it pure} if all the entries in the commutator are elements of the set $X$. A polynomial in $K\langle X \rangle$ is called {\it proper} if it is a linear combination of products of pure commutators. One important property of proper polynomials is that all identities of an algebra with a unit follow from its proper ones (\cite{D1999}, Proposition 4.3.3). We denote by $P_n$ the vector space of multilinear polynomials of degree $n$ in  $K\langle x_1, \dots, x_n \rangle$ and by $\Gamma_n$ the subspace of $P_n$ of proper polynomials. 

If $\mathfrak{V}$ is a variety of algebras, we denote
\[
P_n(\mathfrak{V}) = \frac{P_n}{P_n \cap \mathrm{Id}(\mathfrak{V})} 
\quad \text{ and } \quad \Gamma_n(\mathfrak{V}) = \frac{\Gamma_n}{\Gamma_n \cap \mathrm{Id}(\mathfrak{V})}.
\]

To any polynomial from $K\langle X \rangle$ one can associate its complete multilinearization, which is a polynomial in $P_n$. It is well known that for an algebra over a field of characteristic zero, a polynomial is an identity if and only if its complete multilinearization is also an identity.
Hence, in the study of polynomial identities, we can restrict ourselves to proper multilinear polynomials. 

The spaces $P_n(\mathfrak{V})$ and $\Gamma_n(\mathfrak{V})$ are naturally left $S_n$-modules, where $S_n$ denotes the symmetric group in $n$ variables and the action is given by permuting the variables. The representation theory of $S_n$ provides a useful tool in the study of $P_n(\mathfrak{V})$ and $\Gamma_n(\mathfrak{V})$ (see, e.g., \cite{D1999}, Chapter 12). The $S_n$-module structure of $\Gamma_n(\mathfrak{N}_p)$ for $n=3$ is described in \cite{Volichenko_preprint} and for $n=4$ in \cite{Stoyanova1984(N4)}. We recall these results in Sections \ref{sec_N3} and \ref{sec_N4} and use them in the description of bases for the polynomial identities of $F_n(\mathfrak{N}_3)$ and $F_n(\mathfrak{N}_4)$. 

A useful tool to show that a polynomial $f \in K\langle X \rangle$ is not an identity for $F_n(\mathfrak{N}_p)$ for even values of $p$ is the following result of Deryabina and Krasilnikov \cite{DK2017}. Let $E$ denote the Grassmann algebra over a countable dimensional vector space and let $E_r$ denote the Grassmann algebra over an $r$-dimensional vector space. Then the following proposition holds.

\begin{prop} \cite{DK2017} \label{prop_DeryabinaKrasilnikov}
 For any integer $k \geq 1$, the algebra $E\otimes E_{2k}$ satisfies the polynomial identity $[x_1, \dots, x_{2k+3}] = 0$.
\end{prop}

In the end of this section, we recall Lemma 8 of \cite{GdeM2015}:

\begin{lemma} \cite{GdeM2015} \label{nvariables}
    Let $n\geq 2$ be an integer and $A$ an arbitrary associative algebra. A polynomial $f\in K\langle x_1, \dots, x_n\rangle $ is a polynomial identity for $F_n(A)$ if and only if $f$ is a polynomial identity for $A$.
\end{lemma}

The above result means that in order to obtain a polynomial identity for a relatively free algebra of rank $n$ in a variety $\mathfrak V$ which is not in $\mathrm{Id}(\mathfrak V)$, one needs to look for polynomials with more than $n$ variables.

\section{The variety $\mathfrak{N}_3$} \label{sec_N3}

In this section, we consider the variety $\mathfrak{N_3}$ defined by the polynomial identity $[x_1, x_2, x_3, x_4] = 0$. This identity can be written as
$[x_1,x_2,x_3]x_4 = x_4[x_1,x_2,x_3]$, which means: triple commutators are central.

We start with some consequences of $\mathfrak N_3$. They are immediate corollaries of Theorem \ref{thm_products_of_commutators} for $p = 3$ and can also be found in Gordienko's paper \cite{Gordienko2007} for $\operatorname{char}{K} \neq 3$.

\begin{lemma} \label{consequencesN3}
    The identity $[x_1,x_2,x_3, x_4]  = 0$ has the following consequences:
\begin{enumerate}
    \item $[[x_1, x_2], [x_3,x_4]] = 0$.
    \item $[x_1, x_2][x_3,x_4,x_5] = 0$.
    \item $[x_1, x_2][x_3,x_4][x_5, x_6] + [x_1,x_2][x_3,x_6][x_5,x_4] = 0$.
\end{enumerate}
\end{lemma}

Using identities 1 and 3 above, we obtain:

\begin{lemma}\label{permuting}
    If $m\geq 3$, then we have
    \[[x_{\sigma(1)}, x_{\sigma(2)}][x_{\sigma(3)}, x_{\sigma(4)}] \cdots [x_{\sigma(2m-1)}, x_{\sigma(2m)}] \equiv_{I_4} (-1)^{\sigma}[x_1,x_2][x_3,x_4]\cdots [x_{2m-1},x_{2m}].\]
\end{lemma}

In addition, the following is an immediate consequence of identity $3$ of Lemma \ref{consequencesN3}.

\begin{lemma}\label{repeated}
    If $m\geq 3$, \[[x_{i_1},x_{i_2}]\cdots [x_{i_{2m-1}},x_{i_{2m}}] = 0\] is a consequence of $[x_1,x_2,x_3,x_4]$ whenever there is a pair of repeated variables.
\end{lemma}

The following proposition is a consequence of Lemma \ref{permuting} and of Lemma \ref{repeated}.

\begin{prop}\label{identity}
    Let $m\geq 5$ be an odd integer. If $n \leq  m$, then $F_n(\mathfrak{N}_3)$ satisfies the polynomial identity
    \[f=[x_1,x_2]\cdots [x_{m},x_{m+1}] = 0.\]

\end{prop}

	\begin{proof}
		Notice that it is enough to prove that if $n$ is odd, then $F_n(\mathfrak{N}_3)$  satisfies $[x_1, x_2]\cdots [x_n,x_{n+1}]$. This will imply that if $n$ is even, then $[x_1,x_2]\cdots [x_{n+1},x_{n+2}]$ is an identity for $F_n(\mathfrak{N}_3)$.
		
		So assume $n$ is odd and let $m_1,\dots, m_{n+1}$ be monomials in $K\langle x_1,\dots,x_n\rangle$. We need to show that $f(m_1,\dots,m_{n+1}) = [m_1,m_2]\cdots[m_{n},m_{n+1}] \in I_4$.
		
		Using the identity $[ab,c] = a[b,c]+[a,c]b$, we obtain that $f(m_1,\dots,m_{n+1})$ is a linear combination of elements of type 
		\[
		h = u_0[x_{i_1}, x_{i_2}]u_1[x_{i_3}, x_{i_4}]u_2 \cdots u_{t-1}[x_{i_{n}}, x_{i_{n+1}}] u_t,
		\]
		where $u_0, \dots, u_t$ are monomials in $x_1,\dots, x_n$.
		
		Since any product of a double and a triple commutator is an element of $I_4$, we use the identity \[c[a,b] = [a,b]c-[a,b,c],\] to move monomials to the rightmost position and we obtain that modulo $I_4$, 
		\[
		h = [x_{i_1}, x_{i_2}]\cdots [x_{i_{n}}, x_{i_{n+1}}] u_0\cdots u_t,
		\] 
		Now, since we consider only $n$ variables, there is at least one repetition among $x_{i_1}, \dots , x_{i_{n+1}}$. Now the proof follows from Lemma \ref{repeated}.	
	\end{proof}
    
We are now ready to prove the following theorem.

\begin{theorem} \label{thm_identities_N3}
Consider the relatively free algebra $F_n(\mathfrak{N_3})$.
\begin{itemize}

\item[(i)] Let $n \geq 4$ be an even integer. Then the $T$-ideal of polynomial identities of $F_n(\mathfrak{N_3})$ is generated by the polynomials
\begin{align*}
[x_1, x_2, x_3, x_4] \text{ and } [x_1, x_2] \cdots [x_{n+1}, x_{n+2}].
\end{align*}

\item[(ii)] Let $n \geq 5$ be an odd integer. Then the $T$-ideal of polynomial identities of $F_n(\mathfrak{N_3})$ is generated by the polynomials
\begin{align*}
[x_1, x_2, x_3, x_4] \text{ and } [x_1, x_2] \cdots [x_n, x_{n+1}].
\end{align*}
\end{itemize}
\end{theorem}

\begin{proof}
Let $n \geq 4$ be an even integer. Proposition \ref{identity} implies that the product $[x_1, x_2] \cdots [x_{n+1}, x_{n+2}]$ is a polynomial identity for $F_n(\mathfrak{N_3})$. Therefore, to prove the statement, it is enough to show that if $f$ is a proper polynomial in $K \left \langle x_1, \dots, x_m \right \rangle$ for some integer $m \geq 1$ and if $f$ is an identity for $F_n(\mathfrak{N_3})$, then $f \in \left \langle [x_1, x_2, x_3, x_4], [x_1, x_2] \cdots [x_{n+1}, x_{n+2}] \right \rangle ^T$. 

Using Theorem \ref{thm_products_of_commutators}, we can write the polynomial $f$ in the form
\begin{align*}
f = &\sum_{k \geq 1} \sum_{1\leq i_1, i_2, \dots, i_{2k} \leq m} a_{i_1,i_2\dots, i_{2k}}[x_{i_1}, x_{i_2}][x_{i_3}, x_{i_4}]\cdots [x_{i_{2k-1}}, x_{i_{2k}}] +\\
& \sum_{1 \leq i_1, i_2, i_3 \leq m} b_{i_1, i_2, i_3} [x_{i_1}, x_{i_2}, x_{i_3}] + f_1,
\end{align*}
where $f_1 \in I_4$. Since $f$ is an identity for $F_n(\mathfrak{N_3})$, each multihomogeneous component of $f$ is also an identity for $F_n(\mathfrak{N_3})$. By Lemma \ref{nvariables}, all multihomogeneous components of degree less or equal to $n$ are identities of $F_n(\mathfrak{N_3})$ if and only if they belong to $I_4$. Therefore, we only need to consider multihomogeneous components of degree greater than $n$. Hence, we need to consider
\[
f' = \sum_{k > n/2} \sum_{1\leq i_1, i_2, \dots, i_{2k} \leq m} a_{i_1,i_2\dots, i_{2k}}[x_{i_1}, x_{i_2}][x_{i_3}, x_{i_4}]\cdots [x_{i_{2k-1}}, x_{i_{2k}}].
\]
We notice that each monomial of $f'$ belongs to $\left\langle [x_1, x_2] \cdots [x_{n+1}, x_{n+2}]\right\rangle^T$. This finishes the proof of the theorem for even $n$.

Next, let $n \geq 5$ be an odd integer and let $f$ be a polynomial identity for $F_n(\mathfrak{N}_3)$. Then $f$ is also a polynomial identity for $F_{n-1}(\mathfrak{N}_3)$ and hence by Part (i)
\[
f \in \left\langle [x_1, x_2, x_3, x_4], [x_1, x_2]\cdots [x_{n}, x_{n+1}]\right\rangle^T.
\]
\end{proof}
\begin{remark} We notice that the polynomial $[x_1, x_2][x_3, x_4]$ is not an identity for $F_2(\mathfrak{N}_3)$, nor for $F_3(\mathfrak{N}_3)$. Hence, Theorem \ref{thm_identities_N3} does not hold for $n=2$ and $n=3$.
\end{remark}

The next step is to consider the cases $n=2$ and $n=3$.
By $s_m(x_1, \dots, x_m)$ we denote the standard polynomial of degree $m$, i.e.,
\[
s_m(x_1, \dots, x_m) = \sum_{\sigma \in S_m} (-1)^{\sigma}x_{\sigma(1)} x_{\sigma(2)} \cdots x_{\sigma(m)}.
\]

It is well-known that when $m = 2k$ we have
\[
s_{2k}(x_1, \dots, x_k) = \frac{1}{2^k}\sum_{\sigma \in S_{2k}} (-1)^{\sigma} [x_{\sigma(1)}, x_{\sigma(2)}] \cdots [x_{\sigma(2k-1)}, x_{\sigma(2k)}].
\]

Then, the following proposition holds.

\begin{prop} \label{standard4}
Let $n=2$ or $n=3$. Then the standard polynomial 
\[
s_4(x_1, \dots, x_4) =\frac{1}{2^{2}}\sum_{\sigma \in S_{4}} (-1)^{\sigma} [x_{\sigma(1)}, x_{\sigma(2)}] [x_{\sigma(3)}, x_{\sigma(4)}]
\]
is an identity for $F_n(\mathfrak{N}_3)$.
\end{prop}

\begin{proof}
The proof is similar to the proofs of Propositions  \ref{prop_polynomialIdentity_odd_n} and \ref{prop_identity_g3_odd}. The idea is to show that any evaluation of $s_4$ on monomials vanishes. Let $m_1, m_2, m_3, m_4$ be monomials in $K \left \langle x_1, \dots, x_n \right \rangle$ (where $n =2$ or $n=3$).
We set $m_i = x_i^1 \cdots x_i^{q_i}$ for each $i = 1, \dots, 4$, where $x_i^j \in \{x_1, \dots, x_n\}$ for each $i$ and $j$. Then, we have to show that
\[
s_4(m_1, \dots, m_{4}) \equiv_{I_4} 0,
\]
where the notation $\equiv_{I_4}$ means equivalence modulo $I_4$.
We use Lemma \ref{CommutatorOfMonomials} to move variables outside the commutators. Then due to Theorem \ref{thm_products_of_commutators} (vi) we observe that up to an element of $I_4$ we can exchange the places of variables outside commutators and we can move the variables which are outside commutators to the left. Thus we obtain
\begin{align*}
s_4(m_1, \dots, m_4) \equiv_{I_4} &\frac{1}{4} \sum_{i_1, \dots, i_4}^{q_1, \dots, q_4}  \beta_{i_1, \dots, i_4}  \\
&\sum_{\sigma \in S_4} (-1)^{\sigma}[x_{\sigma(1)}^{i_{\sigma(1)}}, x_{\sigma(2)}^{i_{\sigma(2)}}] [x_{\sigma(3)}^{i_{\sigma(3)}}, x_{\sigma(4)}^{i_{\sigma(4)}}],
\end{align*}
where
\[
\beta_{i_1, \dots, i_4} = x_{1}^1\cdots \widehat{x_{1}^{i_1}} \cdots x_{1}^{q_{1}} \cdots x_{4}^1\cdots \widehat{x_{4}^{i_{4}}} \cdots x_{4}^{q_{4}}.
\]
Here, the notation $\widehat{x_{1}^{i_1}}$ means that the element $x_{1}^{i_1}$ is missing from the product.

Since $n=2$ or $n=3$, we have repeating variables in the above equation. This implies that the right-hand side of the above equation vanishes. This proves the statement.
\end{proof}

Next, we recall the following result due to Volichenko.

\begin{prop} \label{Volichenko} \cite{Volichenko_preprint}  The $S_m$-module $\Gamma_m(\mathfrak N_3)$ decomposes as a direct sum of irreducible $S_m$-modules pairwise non-isomorphic in the following way:

\begin{enumerate}
    \item $\Gamma_2(\mathfrak{N}_3)=(KS_2) \cdot u_2^{(1)}$.
    \item $\Gamma_3(\mathfrak{N}_3)=(KS_3)\cdot u_3^{(1)}$.
    \item $\Gamma_4(\mathfrak{N}_3) = (KS_4)\cdot u_{4,2}^{(2)} \oplus (KS_4)\cdot u_{4,4}^{(5)}$
    \item If $m\geq 5$, then $\Gamma_m(\mathfrak N_3)$ decomposes as 
    \[\Gamma_m(\mathfrak{N}_3) = \left\{\begin{array}{cc}
        K(S_m)\cdot u_{(m,m)}^{(5)} & \text{if } m \text{ is even}  \\
         0 & \text{if } m \text{ is odd}  
    \end{array}\right.\]
\end{enumerate}
Here the polynomials $u_2^{(1)}$, $u_3^{(1)}$, $u_{4,2}^{(2)}$, $u_{m,m}^{(5)}$ are the complete linearizations of the following polynomials:

\begin{enumerate}
 \item $f_2^{(1)}=[x_1,x_2]$.
 \item $f_3^{(1)}=[x_2,x_1,x_1]$.
 \item $f_{4,2}^{(2)}=[x_1,x_2]^2$.
 \item $f_{m,m}^{(5)}=s_m(x_1,\dots,x_m)$.
 \end{enumerate}

\end{prop}

The above result was proven in the unpublished preprint \cite{Volichenko_preprint}. 
A proof of the above result can be obtained as follows (we just include a sketch of the proof, not to make the paper too long).

\begin{enumerate}
    \item Notice that the variety $\mathfrak{N}_3$ is a subvariety of the varitey $\mathfrak{M}_5$, generated by the identity $[x_1,x_2][x_3,x_4,x_5]$.
    \item The $S_n$-module structure of the proper multilinear polynomials of degree $n$ modulo the identity $[x_1,x_2][x_3,x_4,x_5]$ is given by Stoyanova-Venkova in \cite{Stoyanova1982(M5)} as a direct sum of irreducible submodules, each of which is generated by a single polynomial.
    \item Verify which of the generators of the modules from (2) are consequences of $[x_1,x_2,x_3,x_4]$ and which are not.
    \item  The $S_n$-module structure of the proper multilinear polynomials of degree $n$ modulo the identity $[x_1,x_2,x_3,x_4]$ will be given by the sum of the irreducible components from (2) whose generators are not consequences of $[x_1,x_2,x_3,x_4]$. 
\end{enumerate}

Now we can use the above result to prove the analogue of Theorem \ref{thm_identities_N3} for $n=2$ and $n=3$.

\begin{theorem}\label{thm_id_small_deg}
    \[\mathrm{Id}(F_2(\mathfrak{N}_3)) = \mathrm{Id}(F_3(\mathfrak{N}_3)) = \langle [x_1,x_2,x_3,x_4], s_4(x_1,x_2,x_3,x_4) \rangle^T.\]
\end{theorem}

\begin{proof}
    Notice that $\mathrm{Id}(F_k(\mathfrak{N}_3))\subseteq \mathrm{Id}(\mathfrak{N}_3)$. Then, the $S_m$-module structure of $\Gamma_m(F_k(\mathfrak{N}_3))$ can be obtained by checking which of the generators of $\Gamma_m(\mathfrak{N}_3)$ are identities for $F_k(\mathfrak{N}_3)$. Recall that a polynomial $f$ is an identity for an algebra $A$ over a field of characteristic zero if and only if its complete linearization is an identity for $A$. Notice that the polynomials $f_2^{(1)}$, $f_{3}^{(1)}$ and $f_{4,2}^{(2)}$ are polynomials in 2 variables. So it follows from Lemma \ref{nvariables} that they are not identities for $F_k(\mathfrak{N}_3)$ if $k\geq 2$. Now it is enough to check if $s_m(x_1,\dots,x_m)$ are identities for $F_k(\mathfrak{N}_3)$. By Proposition \ref{standard4}  $s_4$ is an identity for $F_2(\mathfrak{N}_3)$ and $F_3(\mathfrak{N}_3)$. And since $s_{m+1}$ is a consequence of $s_m$, for each $m$, we obtain that $s_m \in \mathrm{Id}(F_2(\mathfrak{N}_3))$ and $s_m \in \mathrm{Id}(F_3(\mathfrak{N}_3))$ for all $m\geq 4$. This means that for the generators of the irreducible $S_m$-submodules of $\Gamma_m(\mathfrak{N}_3)$, it holds that either they are not identities for $F_k(\mathfrak{N}_3)$ ($k=2$ or $k=3$), or they are consequences of $s_4$. This proves the theorem.  
\end{proof}

\begin{remark}
    We notice that the above approach can also be used to obtain the T-ideal $\mathrm{Id}(F_k(\mathfrak{N}_3))$ for any $k\geq 2$, obtaining a new proof for Theorem \ref{thm_identities_N3}. We decided to present the previous proof, to show an alternative way to obtain the result.
\end{remark}

\begin{remark}
    Notice that Lemma \ref{permuting} implies that we can replace the products of the simple commutators in Theorem \ref{thm_identities_N3} by the standard polynomials of the same degree. As a consequence, we obtain a unified version of Theorems \ref{thm_identities_N3} and \ref{thm_id_small_deg}, as follows. 
\end{remark}

\begin{theorem}
    If $n \geq 2$ is even \[\mathrm{Id}(F_n(\mathfrak N_3)) = \langle [x_1,x_2,x_3,x_4], s_{n+2}(x_1,\dots,x_{n+2})\rangle^T.\]
    Moreover, $\mathrm{Id}(F_{n+1}(\mathfrak N_3)) = \mathrm{Id}(F_{n}(\mathfrak N_3))$.
\end{theorem}

\section{The variety $\mathfrak{N_4}$} \label{sec_N4}
In this section we consider the variety of Lie nilpotent associative algebras defined by the identity $[x_1, x_2, x_3, x_4, x_5] = 0$. We start with some consequences of this polynomial. The following lemma is a direct corollary of Theorem \ref{thm_products_of_commutators} and Lemma \ref{lemma_even_comms} for $p = 4$.

\begin{lemma}\label{consequencesN4}
    The following identities are consequences of $[x_1,x_2,x_3,x_4,x_5] = 0$.
    \begin{enumerate}[(i)]
        \item $[x_1,x_2,x_3,x_4]x_5 - x_5[x_1,x_2,x_3,x_4] = 0$.
        \item $[[x_1,x_2,x_3],[x_4,x_5]] = 0$.
        \item $[x_1,x_2,x_3][x_4,x_5,x_6] = 0$
        \item $[x_1,x_2][x_3,x_4,x_5,x_6] + [x_6,x_2][x_3,x_4,x_5,x_1] = 0$
        \item $[x_1,x_2][x_3,x_4,x_5,x_1] = 0$ 
    \end{enumerate}
\end{lemma}

To prove our main result we will need other consequences of $[x_1,x_2,x_3,x_4,x_5]$. These will be given in the next Lemma.

\begin{lemma}\label{consequencesN4(2)}  
The following identities are consequences of    $[x_1,x_2,x_3,x_4,x_5] = 0$.
\begin{enumerate}
    \item [(vi)] $[x_1,x_2][x_3,x_4][x_2,x_1,x_1] = 0$.
    \item [(vii)] $[x_1,x_2][x_3,x_4][x_1,x_2,x_3] = 0$.
\end{enumerate}
\end{lemma}

\begin{proof}
    
    In identity $(v)$ of Lemma \ref{consequencesN4} we replace $x_4$ by $x_4x_6$. Then using the identity $[ab,c] = a[b,c]+[a,c]b$, and identities (ii), (iii), and (v) we obtain the following consequence

    \begin{align}\label{conseqdeg7}
        0  = & [x_1,x_2][x_6,x_5][x_3,x_4,x_1] + [x_1,x_2][x_3,x_4][x_6,x_5,x_1] \\ + \nonumber & [x_1,x_2][x_4,x_5][x_3,x_6,x_1] + [x_1,x_2][x_3,x_6][x_4,x_5,x_1] 
    \end{align}
    Now we replace $x_6 $ by $x_1$ and use identity (iii) of the above lemma ($I_3I_3\subseteq I_5$) to obtain
    \begin{align*}
        0 = [x_1,x_2][x_3,x_4][x_5,x_1,x_1] + [x_1,x_2][x_5,x_4][x_3,x_1,x_1]
    \end{align*}
    Finally, we replace $x_5$ by $x_2$ and obtain
    \begin{align*}
        [x_1,x_2][x_3,x_4][x_2,x_1,x_1] = 0
    \end{align*}
    This is identity (vi) above.

    To prove identity (vii), we replace in identity (\ref{conseqdeg7}) $x_6$ by $x_3$. Again, using identities from Lemma \ref{consequencesN4}, we obtain
    \begin{align*}
        0 = [x_1,x_2][x_3,x_5][x_3,x_4,x_1] + [x_1,x_2][x_3,x_4][x_3,x_5,x_1] = 0.
    \end{align*}
    Finally, replacing $x_5$ by $x_2$, and using identity (iii) of Lemma \ref{consequencesN4}, we obtain
    \begin{align*}
        [x_1,x_2][x_3,x_4][x_3,x_2,x_1] = 0
    \end{align*}
    The Lemma is proved.
\end{proof}

The main goal of this section is to prove the following statement.

\begin{theorem} \label{thm_PI} Consider the algebra $F_n(\mathfrak{N}_4)$, which is the relatively free algebra of rank $n$ in the variety satisfying the polynomial identity $[x_1, \dots, x_5]$.
\begin{itemize}
\item[(i)]
Let $n \geq 5$ be an odd integer. Then the $T$-ideal of polynomial identities of $F_n(\mathfrak{N}_4)$ is generated by
\begin{enumerate}
\item $[x_1, x_2, x_3, x_4, x_5]$,
\item $ [x_1, x_2] \cdots [x_n, x_{n+1}][x_{n+2}, x_{n+3}]$,
\item $\sum_{\sigma \in S_{n+1}}(-1)^{\sigma} [x_{\sigma(1)}, x_{\sigma(2)}]\cdots [x_{\sigma(n-2)}, x_{\sigma(n-1)}][x_{\sigma(n)}, x_{\sigma(n+1)}, x_1]$.
\end{enumerate}
Another equivalent set of generators for the $T$-ideal of identities of $F_n(\mathfrak{N}_4)$ is given by replacing the polynomial $[x_1, x_2] \cdots [x_n, x_{n+1}][x_{n+2}, x_{n+3}]$ by the standard polynomial $s_{n+3} = \sum_{\sigma\in S_{n+3}}(-1)^\sigma x_{\sigma(1)}\cdots x_{\sigma(n+3)}$.
\item[(ii)] Similarly, if $n\geq 4$ is an even number, then the $T$-ideal of polynomial identities of $F_n(\mathfrak{N}_4)$ is generated by
\begin{enumerate}
    \item $[x_1, x_2, x_3, x_4, x_5]$, 
    \item $s_{n+2}(x_1, \dots, x_{n+2}) = \sum_{\sigma\in S_{n+2}}(-1)^\sigma x_{\sigma(1)}\cdots x_{\sigma(n+2)}$
    \item $\sum_{\sigma \in S_{n+1}}(-1)^{\sigma} [x_{\sigma(1)}, x_{\sigma(2)}]\cdots [x_{\sigma(n-3)}, x_{\sigma(n-2)}][[x_{\sigma(n-1)}, x_{\sigma(n)}],[x_{\sigma(n+1)}, x_1]]$. 
\end{enumerate}
\end{itemize}
\end{theorem}

To prove this theorem we use a similar approach to that of \cite{GdeM2015} and as in the proof of Theorem \ref{thm_id_small_deg}. First, we notice that the $S_m$-module structure of $\Gamma_m(\mathfrak{N}_4)$ was determined by Stoyanova-Venkova in the paper \cite{Stoyanova1984(N4)}. Here we recall her result.

First, we define the following polynomials.
\begin{align*}
&g_1^{(2k-1)} = \sum_{\sigma \in S_{2k-3}}(-1)^{\sigma} [x_{\sigma(1)}, x_{\sigma(2)}]\cdots [x_{\sigma(2k-5)}, x_{\sigma(2k-4)}][x_{\sigma(2k-3)}, x_1, x_1]; \\
&g_2^{(2k-1)} = \sum_{\sigma \in S_{2k-3}}(-1)^{\sigma} [x_{\sigma(1)}, x_{\sigma(2)}]\cdots [x_{\sigma(2k-5)}, x_{\sigma(2k-4)}][x_2, x_1, x_{\sigma(2k-3)}]; \\
&g_3^{(2k-1)} = \sum_{\sigma \in S_{2k-2}}(-1)^{\sigma} [x_{\sigma(1)}, x_{\sigma(2)}]\cdots [x_{\sigma(2k-5)}, x_{\sigma(2k-4)}][x_{\sigma(2k-3)}, x_{\sigma(2k-2)}, x_1]; \\
&g_0^{(5)} = [x_2,x_1][x_2, x_1, x_1];
\end{align*}

\begin{align*}
&g_1^{(2k)} = \sum_{\sigma \in S_{2k-2}}(-1)^{\sigma} [x_{\sigma(1)}, x_{\sigma(2)}]\cdots [x_{\sigma(2k-5)}, x_{\sigma(2k-4)}][x_{\sigma(2k-3)}, x_{\sigma(2k-2)}, x_1, x_1]; \\
&g_2^{(2k)} = [x_1, x_2]s_{2k-2}(x_1, x_2, \dots, x_{2k-2}); \\
&g_3^{(2k)} = \sum_{\sigma \in S_{2k-1}}(-1)^{\sigma} [x_{\sigma(1)}, x_{\sigma(2)}]\cdots [x_{\sigma(2k-5)}, x_{\sigma(2k-4)}][x_{\sigma(2k-3)}, x_{\sigma(2k-2)}, [x_{\sigma(2k-1)}, x_1]]; \\
&g_0^{(6)} = [x_1,x_2]^3;
\end{align*}

Let $f_i^{(j)}$ denote the complete multilinearization of $g_i^{(j)}$.
Then each $f_i^{(j)}$ defines an irreducible $S_j$-module. The following theorem can be found in \cite{Stoyanova1984(N4)}.

\begin{theorem} \cite{Stoyanova1984(N4)} \label{thm_Stoyanova-Venkova}
$\Gamma_m(\mathfrak{N}_4)$ decomposes into a direct sum of irreducible $S_m$-modules in the following way:
\begin{align*}
&\Gamma_2(\mathfrak{N}_4) = \Gamma_2 = Ks_2(x_1, x_2); \\
&\Gamma_3(\mathfrak{N_4}) = \Gamma_3 = K S_3 f_3^{(3)};\\
&\Gamma_4(\mathfrak{N}_4) = \Gamma_4 = K S_4 f_1^{(4)} + K S_4 f_2^{(4)} + K S_4 f_3^{(4)} + K s_4(x_1, \dots, x_4);\\
&\Gamma_5(\mathfrak{N}_4) = K S_5 f_0^{(5)} + K S_5 f_1^{(5)} + K S_5 f_2^{(5)} + K S_5 f_3^{(5)};\\
&\Gamma_6(\mathfrak{N}_4)  =K S_6 f_0^{(6)} + K S_6 f_1^{(6)} + K S_6 f_2^{(6)} + K S_6 f_3^{(6)} + K s_6(x_1, \dots, x_6);
\end{align*}
For $k \geq 4$ we have
\begin{align*}
 &\Gamma_{2k-1}(\mathfrak{N}_4) =  K S_{2k-1} f_1^{(2k-1)} + K S_{2k-1} f_2^{(2k-1)} + K S_{2k-1} f_3^{(2k-1)};\\
&\Gamma_{2k}(\mathfrak{N}_4) = K S_{2k} f_1^{(2k)} + K S_{2k} f_2^{(2k)} + K S_{2k} f_3^{(2k)} + K s_{2k}(x_1, \dots, x_{2k});   
\end{align*}
\end{theorem}

\begin{coro}[\cite{Stoyanova1984(N4)}]\label{Stoyanova-Venkova_corollary}
    The variety $\mathfrak{N}_4$ is generated by the algebras $E\otimes E_2$ and $E_2\otimes E_2\otimes E_2$.
\end{coro}

\begin{proof}
    One needs to check that the generators $g_1^{(k)}, g_2^{(k)}, g_3^{(k)}$ of $\Gamma_m(\mathfrak{N}_4)$ given in Theorem \ref{thm_Stoyanova-Venkova} are not identities for $E\otimes E_2$ and $g_0^{(5)}$ and $g_0^{(6)}$ are not identities for $E_2\otimes E_2\otimes E_2$.
\end{proof}

Now, to prove Theorem \ref{thm_PI} we will consider separately the cases of odd and even $n$.

\subsection{The case of odd $n\geq 5$}

We first consider the case where $n$ is an odd integer greater than or equal to $5$. In the following propositions, we give several polynomials, which are polynomial identities for $F_n(\mathfrak{N}_4)$. 

\begin{lemma} \label{lemma_I5} Let $n \geq 5$ be an odd integer. Whenever $i_1, \dots, i_{n+3} \in \{1, \dots, n\}$, then
\[
[x_{i_1}, x_{i_2}] \cdots [x_{i_{n+2}}, x_{i_{n+3}}] \in I_5.
\]
\end{lemma}

\begin{proof}
Let us denote $h = [x_{i_1}, x_{i_2}][x_{i_3}, x_{i_4}] \cdots [x_{i_{n+2}}, x_{i_{n+3}}]$. Then, there are three repeating variables in $h$  (since $i_1, \dots, i_{n+3} \in \{1,2 \dots, n\}$). We consider two cases.

{\bf Case 1.} The first case is when we can write $h$ in the form $h = h_1h_2$ such that both $h_1$ and $h_2$ are products of double commutators with at least one repeating index. Hence, $h_1 \in I_3$ and $h_2 \in I_3$, which implies that $h \in I_5$.

{\bf Case 2.} The second case is when $h$ cannot be written in the above form. This case splits further into two subcases:

{\bf Case 2.1.} The first case is when $h$ has at least two pairs of commutators with repeating indices in them. In other words, $h$ is of the form
\[
h = s_1[x_{i_1}, x_{i_2}]s_2[x_{i_3}, x_{i_4}]s_3[x_{i_1}, x_{i_5}]s_4[x_{i_3}, x_{i_6}]s_5,
\]
where $s_i$ is a product of double commutators for each $i$. We will first show that up to an element of $I_5$ we can move the commutator $[x_{i_1}, x_{i_5}]$ to the left of all commutators from $s_3$. 
Indeed, let $s_3 = s_3' [y_1, y_2]$.
Then,
\begin{align*}
&h = s_1[x_{i_1}, x_{i_2}]s_2[x_{i_3}, x_{i_4}]s_3'[x_{i_1}, x_{i_5}][y_1, y_2]s_4[x_{i_3}, x_{i_6}]s_5 + \\
&s_1[x_{i_1}, x_{i_2}]s_2[x_{i_3}, x_{i_4}]s_3'[y_1, y_2, x_{i_1}, x_{i_5}]s_4[x_{i_3}, x_{i_6}]s_5 = h_1 + h_2.
\end{align*}

Since commutators of length $4$ are central in $I_5$, we have
\[
h_2 \equiv_{I_5}s_1[x_{i_1}, x_{i_2}]s_2[x_{i_3}, x_{i_4}]s_3's_4[x_{i_3}, x_{i_6}]s_5[y_1, y_2, x_{i_1}, x_{i_5}] \in I_3I_4 \subset I_6.
\]
As before, the notation $\equiv_{I_5}$ means equivalence modulo $I_5$. Therefore, we showed that $h \equiv_{I_5} h_1$.
We continue to move the commutator $[x_{i_1}, x_{i_5}]$ to the left, so that we obtain 
\[
h \equiv_{I_5}s_1[x_{i_1}, x_{i_2}]s_2[x_{i_3}, x_{i_4}][x_{i_1}, x_{i_5}]s_3s_4[x_{i_3}, x_{i_6}]s_5.
\]
Then, we will show that up to an element of $I_5$ we can also switch the places of $[x_{i_1}, x_{i_5}]$ and $[x_{i_3}, x_{i_4}]$. We have

\begin{align*}
h \equiv_{I_5}s_1[x_{i_1}, x_{i_2}]s_2[x_{i_1}, x_{i_5}][x_{i_3}, x_{i_4}]s_3s_4[x_{i_3}, x_{i_6}]s_5 + \\
s_1[x_{i_1}, x_{i_2}]s_2s_3s_4[x_{i_3}, x_{i_6}]s_5[x_{i_3}, x_{i_4},x_{i_1}, x_{i_5}] = h' + h''.
\end{align*}

If there are repeating variables in $s_1[x_{i_1}, x_{i_2}]s_2s_3s_4[x_{i_3}, x_{i_6}]s_5$, then $h'' \in I_3I_4 \subset I_6$ and $h \equiv_{I_5} h'$. If there are no repeating variables in $s_1[x_{i_1}, x_{i_2}]s_2s_3s_4[x_{i_3}, x_{i_6}]s_5$, then all the three repeating indices appear in the commutator of length $4$ and without loss of generality
\[
h'' \equiv_{I_5} s_1[x_{i_1}, x_{i_2}]s_2s_3s_4s_5[y_1, y_2][y_3, y_4,y_5,y_6],
\]
where $y_1, \dots, y_6$ are at most three different variables from the set $\{x_1, \dots, x_n\}$.
By Theorem \ref{thm_products_of_commutators} (vi), $[y_1, y_2][y_3, y_4,y_5,y_6] \in  I_5$ and thus $h'' \in I_5$.

In short, we showed that 
\[
h \equiv_{I_5}s_1[x_{i_1}, x_{i_2}]s_2[x_{i_1}, x_{i_5}][x_{i_3}, x_{i_4}]s_3s_4[x_{i_3}, x_{i_6}]s_5.
\]

Hence, we reduce this case to Case 1 above.

{\bf Case 2.2.} The second case is when $h$ has three repeating indices in a product of three commutators. In other words, $h$ is of the form
\[
h = s_0[x_{i_1}, x_{i_2}]s_1[x_{i_1}, x_{i_3}]s_2[x_{i_2},x_{i_3}][x_{i_4}, x_{i_5}]s_3,
\]
where again $s_i$ is a product of double commutators. Here we use the requirement that $n \geq 5$, i.e., $n+3 \geq 8$ and $f$ has at least $4$ commutators.

We use that $[x_{i_2}, x_{i_3}][x_{i_4},x_{i_5}] \equiv_{I_3} [x_{i_2}, x_{i_4}][x_{i_5},x_{i_3}]$. Hence,
\[
h \equiv_{i_5} s_0[x_{i_1}, x_{i_2}]s_1[x_{i_1}, x_{i_3}]s_2[x_{i_2}, x_{i_4}][x_{i_5},x_{i_3}]s_3,
\]
and this case is reduced to Case 2.1. This completes the proof of the lemma.
\end{proof}

Lemma \ref{lemma_I5} leads to the following proposition.

\begin{prop} \label{prop_polynomialIdentity_odd_n}
Let $n \geq 5$ be an odd integer, and let $A = F_n(\mathfrak{N}_4)$. Then the polynomial
\[
f = [x_1, x_2] \cdots [x_n, x_{n+1}][x_{n+2}, x_{n+3}]
\]
is a polynomial identity for $A$.
\end{prop}

\begin{proof}

Let $m_1, \dots, m_{n+3}$ be monomials in $K\left\langle x_1, \dots, x_n\right\rangle$. We have to show that
\[
f(m_1, \dots, m_{n+3}) \equiv_{I_5} 0.
\]
Using the identity
\begin{align} \label{identity_assoc}
[a, bc] = b[a,c] + [a,b]c
\end{align}

we obtain that $f(m_1, \dots, m_{n+3})$ is a linear combination of elements of the type 
\[
h = u_0[x_{i_1}, x_{i_2}]u_1[x_{i_3}, x_{i_4}]u_2 \cdots u_{t-1}[x_{i_{n+2}}, x_{i_{n+3}}] u_t,
\]
where $u_0, \dots, u_t$ are monomials in $K\left\langle x_1, \dots, x_n\right\rangle$.

We will show first that $h \equiv_{I_5} u_0[x_{i_1}, x_{i_2}][x_{i_3}, x_{i_4}] \cdots [x_{i_{n+2}}, x_{i_{n+3}}]u_1\cdots u_t$.

Using the property
\[
c[a,b] = [a,b]c - [a,b,c]
\]
we start moving all monomials $u_1, \dots, u_{t-1}$ to the right starting with $u_1$. After the first two steps we obtain

\begin{align*}
h= &u_0[x_{i_1}, x_{i_2}]u_1[x_{i_3}, x_{i_4}]u_2 [x_{i_5}, x_{i_6}]u_3 \cdots u_{t-1}[x_{i_{n+2}}, x_{i_{n+3}}]u_t = \\
&u_0[x_{i_1}, x_{i_2}][x_{i_3}, x_{i_4}]u_1u_2 [x_{i_5}, x_{i_6}]u_3 \cdots u_{t-1}[x_{i_{n+2}}, x_{i_{n+3}}]u_t - \\
&-u_0[x_{i_1}, x_{i_2}][x_{i_3}, x_{i_4}, u_1]u_2 [x_{i_5}, x_{i_6}]u_3 \cdots u_{t-1}[x_{i_{n+2}}, x_{i_{n+3}}]u_t = \\
&u_0[x_{i_1}, x_{i_2}][x_{i_3}, x_{i_4}] [x_{i_5}, x_{i_6}]u_1u_2u_3 \cdots u_{t-1}[x_{i_{n+2}}, x_{i_{n+3}}]u_t- \\
&-u_0[x_{i_1}, x_{i_2}][x_{i_3}, x_{i_4}][x_{i_5}, x_{i_6}, u_1u_2]u_3 \cdots u_{t-1}[x_{i_{n+2}}, x_{i_{n+3}}]u_t - \\
& - u_0[x_{i_1}, x_{i_2}][x_{i_3}, x_{i_4}, u_1][x_{i_5}, x_{i_6}]u_2u_3 \cdots u_{t-1}[x_{i_{n+2}}, x_{i_{n+3}}]u_t + \\
&u_0[x_{i_1}, x_{i_2}][x_{i_3}, x_{i_4}, u_1][x_{i_5}, x_{i_6}, u_2]u_3 \cdots u_{t-1}[x_{i_{n+2}}, x_{i_{n+3}}]u_t.
\end{align*}

In the last summand we have a product of two triple commutators. Using the property $I_3I_3 \subset I_5$ we obtain that
\[
u_0[x_{i_1}, x_{i_2}][x_{i_3}, x_{i_4}, u_1][x_{i_5}, x_{i_6}, u_2]u_3 \cdots u_{t-1}[x_{i_{n+2}}, x_{i_{n+3}}]u_t \in I_5.
\]

Hence,
\begin{align*}
h \equiv_{I_5} &u_0[x_{i_1}, x_{i_2}][x_{i_3}, x_{i_4}] [x_{i_5}, x_{i_6}]u_1u_2u_3 \cdots u_{t-1}[x_{i_{n+2}}, x_{i_{n+3}}]u_t- \\
&-u_0[x_{i_1}, x_{i_2}][x_{i_3}, x_{i_4}][x_{i_5}, x_{i_6}, u_1u_2]u_3 \cdots u_{t-1}[x_{i_{n+2}}, x_{i_{n+3}}]u_t - \\
& - u_0[x_{i_1}, x_{i_2}][x_{i_3}, x_{i_4}, u_1][x_{i_5}, x_{i_6}]u_2u_3 \cdots u_{t-1}[x_{i_{n+2}}, x_{i_{n+3}}]u_t.
\end{align*}

We continue in the same way. We notice that at each step $h$ is equivalent mod $I_5$ to a sum of elements where in all summands we have moved the product $u_1\dots u_i$ one step to the right. In addition, at each step there is a summand which contains a product of only double commutators and the other summands contain one triple commutator. All summands which contain at least two triple commutators belong to $I_5$ and can be disregarded since $I_2I_3 \subset I_4$, $I_3I_3 \subset I_5$, $I_3I_4 \subset I_6$, and $I_4I_4 \subset I_6$. Therefore, in the end we obtain

\begin{align*}
h \equiv_{I_5} u_0[x_{i_1}, x_{i_2}][x_{i_3}, x_{i_4}] \cdots [x_{i_{n+2}}, x_{i_{n+3}}]u_1\cdots u_t +h_0,
\end{align*}

where $h_0$ is a sum of elements of the form
\[
u_0[x_{i_1}, x_{i_2}] \cdots [x_{i_{k-2}}, x_{i_{k-1}}][x_{i_k}, x_{i_{k+1}}, u'][x_{i_{k+2}}, x_{i_{k+3}}] \cdots [x_{i_{n+2}}, x_{i_{n+3}}]u_{j+1}\cdots u_t
\]
and $u' = u_1\cdots u_j$ for some $i$. 

We first use the property $[L_2, L_3] \subset L_5$, that is, double and triple commutators commute up to an element of $L_5$. Therefore, the elements 
\[
u_0[x_{i_1}, x_{i_2}][x_{i_3}, x_{i_4}] \cdots [x_{i_k}, x_{i_{k+1}}, u'] \cdots [x_{i_{n+2}}, x_{i_{n+3}}]u_1\cdots u_t
\]
and 
\[
u_0[x_{i_k}, x_{i_{k+1}}, u'][x_{i_1}, x_{i_2}][x_{i_3}, x_{i_4}] \cdots \widehat{[x_{i_k}, x_{i_{k+1}}]} \cdots [x_{i_{n+2}}, x_{i_{n+3}}]u_1\cdots u_t
\]
are equal up to an element of $I_5$.  Here, the notation $\widehat{[x_{i_k}, x_{i_{k+1}}]}$ means that this term is missing from the product. Hence, without loss of generality, $h_0$ is a sum of elements of the form 
\[
u_0[x_{i_k}, x_{i_{k+1}}, u'][x_{i_1}, x_{i_2}][x_{i_3}, x_{i_4}] \cdots \widehat{[x_{i_k}, x_{i_{k+1}}]} \cdots [x_{i_{n+2}}, x_{i_{n+3}}]u_1\cdots u_t.
\]

We have that $i_1, \dots, i_{n+3} \in \{1, 2, \dots, n\}$. Hence in the product
\[
[x_{i_1}, x_{i_2}][x_{i_3}, x_{i_4}] \cdots \widehat{[x_{i_k}, x_{i_{k+1}}]} \cdots [x_{i_{n+2}}, x_{i_{n+3}}]
\]
there is at least one repeating index. Therefore,
\[
[x_{i_1}, x_{i_2}][x_{i_3}, x_{i_4}] \cdots \widehat{[x_{i_k}, x_{i_{k+1}}]} \cdots [x_{i_{n+2}}, x_{i_{n+3}}] \in I_3.
\]

Using again that $I_3I_3 \subset I_5$ we obtain that
\[
u_0[x_{i_k}, x_{i_{k+1}}, u'][x_{i_1}, x_{i_2}][x_{i_3}, x_{i_4}] \cdots \widehat{[x_{i_k}, x_{i_{k+1}}]} \cdots [x_{i_{n+2}}, x_{i_{n+3}}]u_1\cdots u_t \in I_5.
\]

Hence, $h_0 \in I_5$ and
\begin{align*}
h \equiv_{I_5} u_0[x_{i_1}, x_{i_2}][x_{i_3}, x_{i_4}] \cdots [x_{i_{n+2}}, x_{i_{n+3}}]u_1\cdots u_t.
\end{align*}

Let us denote $h' = [x_{i_1}, x_{i_2}][x_{i_3}, x_{i_4}] \cdots [x_{i_{n+2}}, x_{i_{n+3}}]$. Then, there are three repeating indices in $h'$  (since $i_1, \dots, i_{n+3} \in \{1,2 \dots, n\}$). By Lemma \ref{lemma_I5}, we obtain that $h' \in I_5$.

In short, we showed that $h \in I_5$ and therefore $f$ is a polynomial identity for $A$.
\end{proof}

\begin{remark} Notice that, when $n=3$ the polynomial $[x_1, x_2][x_1, x_3][x_2, x_3]$ is not a polynomial identity for $F_3(\mathfrak{N}_4)$. Hence, Lemma \ref{lemma_I5} and Proposition \ref{prop_polynomialIdentity_odd_n} do not hold for $n=3$.
\end{remark}

The next proposition shows that the generator $(3)$ of Theorem \ref{thm_PI} (i) is an identity for $F_n(\mathfrak{N}_4)$.
\begin{prop}\label{prop_identity_g3_odd}
Let $n \geq 5$ be again an odd integer. The polynomial
\[
g_3^{(n+2)} = \sum_{\sigma \in S_{n+1}}(-1)^{\sigma} [x_{\sigma(1)}, x_{\sigma(2)}]\cdots [x_{\sigma(n-2)}, x_{\sigma(n-1)}][x_{\sigma(n)}, x_{\sigma(n+1)}, x_1]
\]
is a polynomial identity for $A = F_n(\mathfrak{N}_4)$.
\end{prop}

\begin{proof}
The approach is similar to the proof of Proposition \ref{prop_polynomialIdentity_odd_n}.
Let $m_1, \dots, m_{n+1}$ be monomials in $K\left\langle x_1, \dots, x_n\right\rangle$. We have to show that
\[
g_3^{(n+2)}(m_1, \dots, m_{n+1}) \equiv_{I_5} 0.
\]
We set $m_i = x_i^1 \cdots x_i^{q_i}$ for each $i = 1, \dots, n+1$, where $x_i^j \in \{x_1, \dots, x_n\}$ for each $i$ and $j$. 

By substituting $m_1, \dots, m_{n+1}$ in $g_3^{(n+2)}$ and using Lemma \ref{CommutatorOfMonomials} (i) and (ii) and Proposition \ref{prop_polynomialIdentity_odd_n} we obtain

\begin{align*}
&g_3^{(n+2)}(m_1, \dots, m_{n+1}) \equiv_{I_5} \sum_{\sigma \in S_{n+1}} \sum_{i_1, \dots, i_{n+1}}^{q_1, \dots, q_{n+1}} \sum_{j=1}^{q_1}(-1)^{\sigma}\\
&x_{\sigma(1)}^1\cdots x_{\sigma(1)}^{i_{\sigma(1)-1}}x_{\sigma(2)}^1\cdots x_{\sigma(2)}^{i_{\sigma(2)-1}}[x_{\sigma(1)}^{i_{\sigma(1)}}, x_{\sigma(2)}^{i_{\sigma(2)}}] x_{\sigma(2)}^{i_{\sigma(2)+1}} \cdots x_{\sigma(2)}^{q_{\sigma(2)}} 
x_{\sigma(1)}^{i_{\sigma(1)+1}} \cdots x_{\sigma(1)}^{q_{\sigma(1)}} \cdots \\
& \cdots \\
&x_{\sigma(n)}^1\cdots x_{\sigma(n)}^{i_{\sigma(n)-1}}x_{\sigma(n+1)}^1\cdots x_{\sigma(n+1)}^{i_{\sigma(n+1)-1}} x_1^1 \cdots x_1^{j-1}
[x_{\sigma(n)}^{i_{\sigma(n)}}, x_{\sigma(n+1)}^{i_{\sigma(n+1)}}, x_1^j] \times\\
&x_1^{j+1} \cdots x_1^{q_1}
x_{\sigma(n+1)}^{i_{\sigma(n+1)+1}} \cdots x_{\sigma(n+1)}^{q_{\sigma(n+1)}} 
x_{\sigma(n)}^{i_{\sigma(n)+1}} \cdots x_{\sigma(n)}^{q_{\sigma(n)}}.
\end{align*}

Let us now fix $i_1, \dots, i_{n+1}$ and $j$ and consider the polynomial
\begin{align*}
&g_{i_1, \dots, i_{n+1}, j} = \sum_{\sigma \in S_{n+1}}  (-1)^{\sigma}\\
&x_{\sigma(1)}^1\cdots x_{\sigma(1)}^{i_{\sigma(1)-1}}x_{\sigma(2)}^1\cdots x_{\sigma(2)}^{i_{\sigma(2)-1}}[x_{\sigma(1)}^{i_{\sigma(1)}}, x_{\sigma(2)}^{i_{\sigma(2)}}] x_{\sigma(2)}^{i_{\sigma(2)+1}} \cdots x_{\sigma(2)}^{q_{\sigma(2)}} 
x_{\sigma(1)}^{i_{\sigma(1)+1}} \cdots x_{\sigma(1)}^{q_{\sigma(1)}} \cdots \\
& \cdots \\
&x_{\sigma(n)}^1\cdots x_{\sigma(n)}^{i_{\sigma(n)-1}}x_{\sigma(n+1)}^1\cdots x_{\sigma(n+1)}^{i_{\sigma(n+1)-1}} x_1^1 \cdots x_1^{j-1}
[x_{\sigma(n)}^{i_{\sigma(n)}}, x_{\sigma(n+1)}^{i_{\sigma(n+1)}}, x_1^j] \times\\
&x_1^{j+1} \cdots x_1^{q_1}
x_{\sigma(n+1)}^{i_{\sigma(n+1)+1}} \cdots x_{\sigma(n+1)}^{q_{\sigma(n+1)}} 
x_{\sigma(n)}^{i_{\sigma(n)+1}} \cdots x_{\sigma(n)}^{q_{\sigma(n)}}.
\end{align*}

Using the property $I_3I_3 \subset I_5$ we can move the variables which are outside the commutators to the left of the double commutators. Furthermore, due to Proposition \ref{prop_polynomialIdentity_odd_n}, we can exchange the places of variables outside the commutators. Therefore, we obtain
\begin{align*}
&g_{i_1, \dots, i_{n+1}, j} \equiv_{I_5} \sum_{\sigma \in S_{n+1}}  (-1)^{\sigma} \alpha_{\sigma} 
[x_{\sigma(1)}^{i_{\sigma(1)}},x_{\sigma(2)}^{i_{\sigma(2)}}]
\cdots [x_{\sigma(n-2)}^{i_{\sigma(n-2)}}, x_{\sigma(n-1)}^{i_{\sigma(n-1)}}] [x_{\sigma(n)}^{i_{\sigma(n)}}, x_{\sigma(n+1)}^{i_{\sigma(n+1)}}, x_1^j]\\
&x_1^{j+1} \cdots x_1^{q_1}x_{\sigma(n+1)}^{i_{\sigma(n+1)+1}} \cdots x_{\sigma(n+1)}^{q_{\sigma(n+1)}} 
x_{\sigma(n)}^{i_{\sigma(n)+1}} \cdots x_{\sigma(n)}^{q_{\sigma(n)}},
\end{align*}
where 
\begin{align*}
  & \alpha_{\sigma} =  x_{\sigma(1)}^1\cdots \widehat{x_{\sigma(1)}^{i_{\sigma(1)}}} \cdots x_{\sigma(1)}^{q_{\sigma(1)}} \cdots x_{\sigma(n-1)}^1\cdots \widehat{x_{\sigma(n-1)}^{i_{\sigma(n-1)}}} \cdots x_{\sigma(n-1)}^{q_{\sigma(n-1)}} \times\\
&x_{\sigma(n)}^1\cdots
x_{\sigma(n)}^{i_{\sigma(n)-1}}x_{\sigma(n+1)}^1\cdots x_{\sigma(n+1)}^{i_{\sigma(n+1)-1}}x_1^1 \cdots x_1^{j-1}.
\end{align*}

Now, notice that we can also move to the left  the variables
\[
x_1^{j+1} \cdots x_1^{q_1}x_{\sigma(n+1)}^{i_{\sigma(n+1)+1}} \cdots x_{\sigma(n+1)}^{q_{\sigma(n+1)}} 
x_{\sigma(n)}^{i_{\sigma(n)+1}} \cdots x_{\sigma(n)}^{q_{\sigma(n)}}.
\]
since the product of the pure commutators is in $I_4$ and elements of $I_4$ are central modulo $I_5$.

In short, keeping in mind that we can switch the places of variables outside commutators, we showed that

\begin{align} \label{eq_g_end_odd}
&g_{i_1, \dots, i_{n+1},j}  \equiv_{I_5} \beta\sum_{\sigma \in S_{n+1}}  (-1)^{\sigma}
[x_{\sigma(1)}^{i_{\sigma(1)}},x_{\sigma(2)}^{i_{\sigma(2)}}]
\cdots [x_{\sigma(n-2)}^{i_{\sigma(n-2)}}, x_{\sigma(n-1)}^{i_{\sigma(n-1)}}] [x_{\sigma(n)}^{i_{\sigma(n)}}, x_{\sigma(n+1)}^{i_{\sigma(n+1)}}, x_1^j],
\end{align}
where
\[
\beta = x_{1}^1\cdots \widehat{x_{1}^{i_1}} \cdots x_{1}^{q_{1}} \cdots x_{n+1}^1\cdots \widehat{x_{n+1}^{i_{n+1}}} \cdots x_{n+1}^{q_{n+1}}x_{1}^1\cdots \widehat{x_{1}^{j}} \cdots x_{1}^{q_{1}}.
\]
Remark that $\beta$ is independent of the permutation $\sigma$, it depends only on the fixed choice of $i_1, \dots, i_{n+1}$ and $j$. In addition, $x^{i_{\sigma(1)}}_{\sigma(1)}, \dots, x^{i_{\sigma(n+1)}}_{\sigma(n+1)} \in \{x_1, \dots, x_n\}$, which implies that at least one of the variables appears twice. Therefore, every summand on the right side of Equation (\ref{eq_g_end_odd}) appears twice with two different signs, hence the right side of Equation (\ref{eq_g_end_odd}) vanishes. This implies that for any choice of $i_1, \dots, i_{n+1}$ and $j$ we have
\[
g_{i_1, \dots, i_{n+1},j}  \equiv_{I_5} 0
\]
and thus $g_3^{(n+2)}(m_1, \dots, m_{n+1}) \equiv_{I_5} 0$.
\end{proof}

In the next lemmas we give several polynomials which are not identities for $F_n(\mathfrak{N}_4)$.

\begin{lemma} \label{lemma_not_identity1}
Let $n$ be an odd integer. The polynomial
\[
f = [x_1, x_2] \cdots [x_n, x_{n+1}]
\]
is not a polynomial identity for $F_n(\mathfrak{N}_4)$.
\end{lemma}
\begin{proof}

The polynomial 
\[
g = [x_1, x_2][x_1, x_3][x_4, x_5] \cdots [x_{n-1}, x_n]
\]
is a consequence of $f$ and thus it is enough to show that $g$ is not a polynomial identity for $F_n(\mathfrak{N}_4)$. In addition, $g \in K\left\langle x_1, \dots, x_n \right\rangle$. Therefore, Lemma \ref{nvariables} implies that $g$ is a polynomial identity for $F_n(\mathfrak{N}_4)$ if and only if $g \in I_5$. Hence, we need to show that $g \notin I_5$.

By Proposition \ref{prop_DeryabinaKrasilnikov} or equivalently by Corollary \ref{Stoyanova-Venkova_corollary}, the algebra $E\otimes E_2$ satisfies the identity $[x_1, \dots,x_5]=0$. Therefore, to prove that $g \notin I_5$ it suffices to show that $g$ is not a polynomial identity for $E\otimes E_2$.
We set
\begin{align*}
&x_1 = e_1\otimes g_1 + e_2 \otimes g_2 + e_3 \otimes 1\\
&x_i = e_{i+2}\otimes 1, \text{ for } i \geq 2.
\end{align*}

By direct computation we show that
\[
[x_1, x_2][x_1,x_3] = - 8e_1e_2e_4e_5\otimes g_1g_2.
\]

Therefore,
\begin{align*}
&g(e_1\otimes g_1 + e_2 \otimes g_2 + e_3 \otimes 1, e_4\otimes 1, \dots, e_{n+2}\otimes 1) = \\
&(-8e_1e_2e_4e_5\otimes g_1g_2)(2e_6e_7\otimes 1)\cdots (2e_{n+1}e_{n+2}\otimes 1) = \\
&-2^{(n+3)/2}e_1e_2e_4e_5e_6 \dots e_{n+2}\otimes g_1g_2 \neq 0.
\end{align*}
This proves the statement.
\end{proof}

\begin{lemma}
Let $n$ be again an odd integer. The polynomial
\[
f = [x_1, x_2]\cdots [x_{n-2}, x_{n-1}][x_n, x_{n+1}, x_1]
\]
is not a polynomial identity for $F_n(\mathfrak{N}_4)$.
\end{lemma}

\begin{proof}
The proof is similar to the proof of Lemma \ref{lemma_not_identity1}. The polynomial
\[
g = [x_1, x_2, x_1][x_1, x_3] \cdots [x_{n-1}, x_n]
\]
is a consequence of $f$ and is a polynomial in $n$ variables. Therefore, it is enough to prove that $g$ is not a polynomial identity for $E\otimes E_2$. We set
\begin{align*}
& x_1 = e_1\otimes g_1 + e_2 \otimes g_2 + e_3 \otimes 1\\
& x_i = e_{i+2}\otimes 1 \text { for } i \geq 2.
\end{align*}

By direct computation we show that
\[
[x_1, x_2, x_1][x_1, x_3] = 8e_1e_2e_3e_4e_5 \otimes g_1g_2.
\]

Hence, after the above evaluation, we obtain
\[
g =  2^{\frac{n+3}{2}}e_1\cdots e_{n+2}\otimes g_1g_2 \neq 0.
\]
\end{proof}

\begin{lemma}\label{lemma_not_identity4}
Let $n$ be an odd integer. Then the standard polynomial
\begin{align*}
    s_{n+1}(x_1, \dots, x_{n+1}) =\frac{1}{2^{(n+1)/2}}\sum_{\sigma \in S_{n+1}} (-1)^{\sigma} [x_{\sigma(1)}, x_{\sigma(2)}] \cdots [x_{\sigma(n)}, x_{\sigma(n+1)}]
\end{align*}
is not a polynomial identity for $F_n(\mathfrak{N}_4)$.
\end{lemma}

\begin{proof}
By \cite{Stoyanova1984(N4)}, the polynomial $g_2^{(n+2)}$ is a consequence of $s_{n+1}$. Moreover, $g_2^{(n+2)} \in K\left\langle x_1, \dots, x_n \right\rangle$, hence it is a polynomial identity for $F_n(\mathfrak{N_4})$ if and only if $g_2^{(n+2)} \in I_5$. The last statement does not hold due to Theorem \ref{thm_Stoyanova-Venkova}. Therefore, $g_2^{(n+2)}$ is not a polynomial identity for $F_n(\mathfrak{N_4})$, which implies that $s_{n+1}$ is not a polynomial identity for $F_n(\mathfrak{N_4})$ either.
\end{proof}

We are now ready to prove Theorem \ref{thm_PI} for odd $n$.

\begin{proof} (of Theorem \ref{thm_PI} for odd $n \geq 5$): Let $m > n$. Theorem \ref{thm_Stoyanova-Venkova} gives the decomposition of $\Gamma_m(\mathfrak{N}_4)$ as an $S_m$-module. We need to check which of the generators of $\Gamma_m(\mathfrak{N}_4)$ as an $S_m$-module are polynomial identities for $F_n(\mathfrak{N}_4)$ and we need to show that all identities belong to the T-ideal
\[\langle [x_1, x_2, x_3, x_4, x_5], [x_1, x_2] \cdots [x_n, x_{n+1}][x_{n+2}, x_{n+3}], g_3^{(n+2)}\rangle^T\] or to the $T$-ideal 
\[\langle [x_1, x_2, x_3, x_4, x_5], s_{n+3}, g_3^{(n+2)}\rangle^T.\]

By the above, we have to check if the following polynomials are identities for $F_n(\mathfrak{N}_4)$: 
\begin{align*}
&g_1^{(2k-1)}, g_2^{(2k-1)}, g_3^{(2k-1)}, \text{ for } 2k-1 > n;\\
&g_1^{(2k)}, g_2^{(2k)}, g_3^{(2k)}, s_{2k}(x_1, \dots, x_{2k}), \text{ for } 2k > n.
\end{align*}

First recall by Proposition \ref{prop_identity_g3_odd} that $g_3^{(2k-1)}$ is an identity if $2k-1 \geq n+2$ (that is if $n\leq 2k-3$). 

Now, we consider the polynomials $g_1^{(2k-1)}$ and $g_2^{(2k-1)}$. If $2k-1 = n+2$, then $g_1^{(n+2)}$, $g_2^{(n+2)} \in K\langle x_1, \dots, x_n\rangle$, and they are not identities for $F_n(\mathfrak{N}_4)$ by Lemma \ref{nvariables}.

If $2k-1 > n+2$ (that is $n< 2k-3$) then Theorem 3 of \cite{Stoyanova1984(N4)} implies that $g_1^{(2k-1)}$ and $g_2^{(2k-1)}$ are consequences of $g_{3}^{(n+2)}$. 

Similarly, consider the polynomials $g_1^{(2k)}$, $g_2^{(2k)}$, and $g_3^{(2k)}$. If $2k = n+1$ then $g_1^{(n+1)}$, $g_2^{(n+1)}$, and $g_3^{(n+1)} \in K\langle x_1, \dots, x_n\rangle$ and are not identities for $F_n(\mathfrak{N}_4)$ by Lemma \ref{nvariables}.

If $2k>n+1$, then again by Theorem 3 of \cite{Stoyanova1984(N4)} we have that $g_1^{(2k)}$, $g_2^{(2k)}$, and $g_3^{(2k)}$ are consequences of $g_{3}^{(n+2)}$.

To finish the proof, we have to consider the standard polynomial $s_{2k}(x_1, \dots, x_{2k})$. When $2k \geq n+3$ then 
$$s_{2k} \in \left \langle [x_1, x_2] \cdots [x_n, x_{n+1}][x_{n+2}, x_{n+3}]\right\rangle^T. $$
In particular, $s_{2k}$ is an identity for $F_n(\mathfrak{N}_4)$ by Proposition \ref{prop_polynomialIdentity_odd_n}.

When $2k = n+1$ then
\begin{align*}
    s_{n+1}(x_1, \dots, x_{n+1}) =\frac{1}{2^{(n+1)/2}}\sum_{\sigma \in S_{n+1}} (-1)^{\sigma} [x_{\sigma(1)}, x_{\sigma(2)}] \cdots [x_{\sigma(n)}, x_{\sigma(n+1)}].
\end{align*}
It follows from Lemma \ref{lemma_not_identity4} that this polynomial is not an identity for $F_n(\mathfrak{N}_4)$. 

The above proves that for odd $n\geq 7$, either the polynomials $g_i^{(r)}$ with $r>n$ are not identities for $F_n(\mathfrak{N}_4)$ or they lie in the T-ideal generated by the polynomials in the statement of Theorem \ref{thm_PI}. In particular, this proves Theorem \ref{thm_PI} for odd $n\geq 7$.

Let us now consider $n=5$. Then, Theorem \ref{thm_Stoyanova-Venkova} implies that in addition to the above polynomials we need to check also whether $g_0^{(6)}$ is an identity for $F_5(\mathfrak{N}_4)$. By definition,
\[
g_0^{(6)} = [x_1, x_2]^3.
\]
This is a polynomial in less than $n=5$ variables, hence by Lemma \ref{nvariables} it is not a polynomial identity for $F_5(\mathfrak{N}_4)$. 
\end{proof}

\begin{remark} One can also check that the polynomials $g_i^{(r)}$ for $r \geq  n+3$ modulo $I_5$ are consequences also of the polynomial $[x_1, x_2]\cdots[x_{n+2}, x_{n+3}]$.
\end{remark}

\subsection {The case of even $n\geq 4$}
The goal of this subsection is to prove Theorem \ref{thm_PI} for even $n$. Similarly to the previous section, in the following lemmas we give several polynomials which are polynomial identities for $F_n(\mathfrak{N}_4)$ and several polynomials which are not polynomial identities for $F_n(\mathfrak{N}_4)$. 

\begin{prop}\label{prop_polynomialIdentity_even_n}
Let $n \geq 4$ be an even integer and consider the algebra $A = F_n(\mathfrak{N}_4)$. Then the polynomial
\[
f = [x_1, x_2] \cdots [x_{n+1}, x_{n+2}][x_{n+3}, x_{n+4}]
\]
is a polynomial identity for $A$.
\end{prop}
\begin{proof}
Let $m = n+1$. Then, $m$ is an odd integer and $m \geq 5$. Thus, by Proposition \ref{prop_polynomialIdentity_odd_n}
\[
f = [x_1, x_2] \cdots [x_{m}, x_{m+1}][x_{m+2}, x_{m+3}]
\]
is a polynomial identity for $F_{m}(\mathfrak{N}_4)$. Hence, it is also a polynomial identity for $F_{m-1}(\mathfrak{N}_4) = F_{n}(\mathfrak{N}_4)$. This proves the statement.
\end{proof}

\begin{remark}
    If $n=2$ the above statement is no longer true. Indeed, $n+4 = 6$ and $[x_1,x_2][x_3,x_4][x_5,x_6]$ is not an identity to $F_2(\mathfrak{N}_4)$, since the polynomial $g_0^{(6)} = [x_1,x_2]^3$ is a consequence of it and is not an identity for $F_2(\mathfrak{N}_4)$.
\end{remark}

\begin{prop}\label{prop_polynomialIdentity2_even_n}Let $n \geq 4$ be an even integer and consider the algebra $A = F_n(\mathfrak{N}_4)$. Then the polynomial
\[
f = [x_1, x_2] \cdots [x_{n-1}, x_{n}][x_{n+1}, x_{n+2},x_{n+3}]
\]
is a polynomial identity for $A$.
\end{prop}

\begin{proof}
    Let $m_1, \dots, m_{n+3}$ be monomials in $K\langle x_1, \dots, x_n\rangle$. We need to show that $f(m_1,\dots,m_{n+3}) = 0$ modulo $I_5$.

    Notice that $[ab,u,v] = a[b,u,v]+[a,v][b,u]+[a,u][b,v] + [a,u,v]b$, for any $a,b,u,v \in K\langle X \rangle$.

    Using this fact and Proposition \ref{prop_polynomialIdentity_even_n}, we obtain that modulo $I_5$,  $f(m_1,\dots,m_{n+3})$ is a linear combination of elements of type
    \[u_0[x_{i_1},x_{i_2}]
    u_1[x_{i_3},x_{i_4}]
    u_2 \cdots u_{\frac{n}{2} -1}
    [x_{i_{n-1}},x_{i_n}]
    u_{\frac{n}{2}}[x_{i_{n+1}},x_{i_{n+2}},x_{i_{n+3}}]u_{\frac{n}{2}+1}.\]
    Notice that up to elements of $I_3$, we can permute each of the simple commutators with the monomial $u_i$. And since $I_3I_3\subseteq I_5$, we obtain that each element of the above type is equivalent modulo $I_5$ to an element of the type 
    \[u_0u_1 \cdots u_{\frac{n}{2}}[x_{i_1},x_{i_2}]\cdots [x_{i_{n-1}},x_{i_n}][x_{i_{n+1}}, x_{i_{n+2}}, x_{i_{n+3}}]u_{\frac{n}{2}+1}.\]
    But $\{i_1, \dots, i_{n+3}\}\subseteq \{1, \dots, n\}$. Then we must have at least three repeated indexes. 
    
    If we have four occurrences of $x_1$, we obtain an element of $I_5$ due to Lemma \ref{consequencesN4}.

    If we have three occurrences of $x_1$ and two occurrences of $x_2$, we obtain an element of $I_5$ due to Lemma \ref{consequencesN4} and identity (vi) of Lemma \ref{consequencesN4(2)}.

    If we have two occurrences of $x_1$,  two occurrences of $x_2$ and  two occurrences of $x_3$, we obtain an element of $I_5$ due to Lemma \ref{consequencesN4} and of identity (vii) of Lemma \ref{consequencesN4(2)}.

    All cases have been considered and the proposition is proved.   
\end{proof}

\begin{remark}
    The above proposition fails for $n=2$. Indeed, if $n=2$, then $n+3 = 5$ and $[x_1,x_2][x_3,x_4,x_5]$ is not an identity for $F_2(\mathfrak{N_4})$, since the polynomial $g_0^{(5)} = [x_2,x_1][x_2,x_1,x_1]$ is a consequence of it and it is not an identity for $F_2(\mathfrak{N}_4)$.
\end{remark}

\begin{lemma}
    Let n be an even integer. Then the polynomial 
    \[f = [x_1,x_2]\cdots[x_{n+1},x_{n+2}]\]
    is not a polynomial identity for $F_n(\mathfrak N_4)$.
\end{lemma}

\begin{proof}
    The polynomial 
    \[g = [x_1,x_2][x_1,x_3][x_1,x_4][x_5,x_6]\cdots [x_{n-1},x_n]\]
    is a consequence of $f$. Once we show that $g$ is not a polynomial identity of $F_n(\mathfrak N_4)$, the proof will be completed.

    Now we proceed as in the proof of Lemma \ref{lemma_not_identity1}. By considering the evaluations in $E\otimes E_2$
    \begin{align*}
        &x_1 = e_1\otimes g_1 + e_2 \otimes g_2 + e_3 \otimes 1\\
        &x_i = e_{i+2}\otimes 1, \text{ for } i \geq 2.
    \end{align*} we obtain 
    $g = -2^{\frac{n+4}{2}}e_1\cdots e_n \otimes g_1g_2$, which is nonzero in $E\otimes E_2$. As a consequence, $f$ is not a polynomial identity for $F_n(\mathfrak N_4)$.    
\end{proof}

The next proposition shows that the generator $(3)$ Theorem \ref{thm_PI} (ii) is an identity for $F_n(\mathfrak{N}_4)$.

\begin{prop} \label{g3n+2_even}
    If $n \geq 4$ is an even integer, then $g_{3}^{(n+2)}$ is a polynomial identity for $F_n(\mathfrak{N}_ 4)$.
\end{prop}

    \begin{proof}
The proof  is similar to the proof of Proposition \ref{prop_identity_g3_odd}.

We first claim that $g_{3}^{(n+2)}$ vanishes under any evaluation by pure variables of $F_n(\mathfrak{N}_ 4)$.

Indeed, notice that $g_3^{(n+2)}$ is a multihomogeneous polynomial in $n+1$ variables of multidegree $(2, 1, \dots, 1)$. Also, it is alternating in the variables $x_2, \dots, x_{n+1}$. In particular, it vanishes whenever we evaluate the variables $x_2, \dots, x_{n+1}$ by  repeated elements. Since in $F_n(\mathfrak{N}_4)$ we have only $n$ variables, any evaluation of the variables of $g_{3}^{(n+2)}$ by pure variables of $F_n(\mathfrak{N}_4)$ will result in a repetition. So it is enough to consider the case in which the repeating variable is $x_1$. Also, without loss of generality, we may assume $x_{n+1}$ is evaluated by $x_1$.

Write $g_{3}^{(n+2)} = \sum_{\sigma\in S_{n+1}} (-1)^{\sigma}p_\sigma$, where 
    \[p_\sigma = [x_{\sigma(1)},x_{\sigma(2)}]\cdots [x_{\sigma(n-3)},x_{\sigma(n-2)}][[x_{\sigma(n-1)},x_{\sigma(n)}][x_{\sigma(n+1)},x_1]].\]

    When substituting $x_{n+1}$ by $x_1$, for each $\sigma \in S_{n+1}$, the permutation $\tau = \sigma(1 \ n+1)$ is such that $(-1)^{\sigma} = - (-1)^{\tau}$ and $p_{\sigma} = p_{\tau}$. As a consequence, after such evaluation, $g_3^{(n+2)}$ vanishes. 

This proves our claim.

Let now $m_1, \dots, m_{n+1}$ be monomials in $K\left\langle x_1, \dots, x_n\right\rangle$. We have to show that
\[
g_3^{(n+2)}(m_1, \dots, m_{n+1}) \equiv_{I_5} 0.
\]
We set $m_i = x_i^1 \cdots x_i^{q_i}$ for each $i = 1, \dots, n+1$, where $x_i^j \in \{x_1, \dots, x_n\}$ for each $i$ and $j$. 

One can prove as a consequence of Lemma \ref{CommutatorOfMonomials} that for any $a_i, b_i, c_i, d_i$ in any associative algebra

\begin{equation}\label{identity_4_assoc}
\begin{aligned} 
 &[[a_1\cdots a_k, b_1 \cdots b_l],[c_1 \cdots c_t,d_1\cdots d_u]] = \\ \sum_{i =1}^k\sum_{j = 1}^l \sum_{r=1}^t\sum_{s=1}^u
 & p(i,j,s)[[a_i, b_j],[ c_r,d_s]]q(i,j,r,s) + f,
\end{aligned}
\end{equation}
where $ p(i,j,r,s) = a_1\cdots a_{i-1} b_1 \cdots b_{j-1}c_1 \cdots c_{r-1}d_1\cdots d_{s-1}$, \\$q(i,j,r,s)d_{s+1}\cdots d_{u}c_{r+1}\cdots c_m b_{j+1}\cdots b_l a_{i+1}\cdots a_k$ and $f \in I_2^3+I_2I_3 \mod I_5$.

By substituting $m_1, \dots, m_{n+1}$ in $g_3^{(n+2)}$ and using Lemma \ref{CommutatorOfMonomials}, Equation (\ref{identity_4_assoc}) and Propositions \ref{prop_polynomialIdentity_even_n} and \ref{prop_polynomialIdentity2_even_n}, we obtain
\begin{align*}
&g_3^{(n+2)}(m_1, \dots, m_{n+1}) \equiv_{I_5} \sum_{\sigma \in S_{n+1}} \sum_{i_1, \dots, i_{n+1}}^{q_1, \dots, q_{n+1}} \sum_{j=1}^{q_1}(-1)^{\sigma}\\
&x_{\sigma(1)}^1\cdots x_{\sigma(1)}^{i_{\sigma(1)-1}}x_{\sigma(2)}^1\cdots x_{\sigma(2)}^{i_{\sigma(2)-1}}[x_{\sigma(1)}^{i_{\sigma(1)}}, x_{\sigma(2)}^{i_{\sigma(2)}}] x_{\sigma(2)}^{i_{\sigma(2)+1}} \cdots x_{\sigma(2)}^{q_{\sigma(2)}} 
x_{\sigma(1)}^{i_{\sigma(1)+1}} \cdots x_{\sigma(1)}^{q_{\sigma(1)}} \times\cdots \\
& \cdots \\
&x_{\sigma(n-1)}^1\cdots x_{\sigma(n-1)}^{i_{\sigma(n-1)-1}}x_{\sigma(n)}^1\cdots x_{\sigma(n)}^{i_{\sigma(n)-1}}x_{\sigma(n+1)}^1\cdots x_{\sigma(n+1)}^{i_{\sigma(n+1)-1}} x_1^1 \cdots x_1^{j-1}
[[x_{\sigma(n-1)}^{i_{\sigma(n-1)}},x_{\sigma(n)}^{i_{\sigma(n)}}],[ x_{\sigma(n+1)}^{i_{\sigma(n+1)}}, x_1^j]] \times\\
&x_1^{j+1} \cdots x_1^{q_1}
x_{\sigma(n+1)}^{i_{\sigma(n+1)+1}} \cdots x_{\sigma(n+1)}^{q_{\sigma(n+1)}} 
x_{\sigma(n)}^{i_{\sigma(n)+1}} \cdots x_{\sigma(n)}^{q_{\sigma(n)}}x_{\sigma(n-1)}^{i_{\sigma(n-1)+1}} \cdots x_{\sigma(n-1)}^{q_{\sigma(n-1)}}.
\end{align*}

Let us now fix $i_1, \dots, i_{n+1}$ and $j$ and consider the polynomial
\begin{align*}
&h_{i_1, \dots, i_{n+1}, j} = \sum_{\sigma \in S_{n+1}}  (-1)^{\sigma}\\
&x_{\sigma(1)}^1\cdots x_{\sigma(1)}^{i_{\sigma(1)-1}}x_{\sigma(2)}^1\cdots x_{\sigma(2)}^{i_{\sigma(2)-1}}[x_{\sigma(1)}^{i_{\sigma(1)}}, x_{\sigma(2)}^{i_{\sigma(2)}}] x_{\sigma(2)}^{i_{\sigma(2)+1}} \cdots x_{\sigma(2)}^{q_{\sigma(2)}} 
x_{\sigma(1)}^{i_{\sigma(1)+1}} \cdots x_{\sigma(1)}^{q_{\sigma(1)}} \cdots \\
& \cdots \\
&x_{\sigma(n-1)}^1\cdots x_{\sigma(n-1)}^{i_{\sigma(n-1)-1}}x_{\sigma(n)}^1\cdots x_{\sigma(n)}^{i_{\sigma(n)-1}}x_{\sigma(n+1)}^1\cdots x_{\sigma(n+1)}^{i_{\sigma(n+1)-1}} x_1^1 \cdots x_1^{j-1}
[[x_{\sigma(n-1)}^{i_{\sigma(n-1)}},x_{\sigma(n)}^{i_{\sigma(n)}}],[ x_{\sigma(n+1)}^{i_{\sigma(n+1)}}, x_1^j]] \times\\
&x_1^{j+1} \cdots x_1^{q_1}
x_{\sigma(n+1)}^{i_{\sigma(n+1)+1}} \cdots x_{\sigma(n+1)}^{q_{\sigma(n+1)}} 
x_{\sigma(n)}^{i_{\sigma(n)+1}} \cdots x_{\sigma(n)}^{q_{\sigma(n)}}x_{\sigma(n-1)}^{i_{\sigma(n-1)+1}} \cdots x_{\sigma(n-1)}^{q_{\sigma(n-1)}}.
\end{align*}

Using the property $I_3I_3 \subset I_5$, up to an element of $I_5$, we can move the variables which are outside the commutators to the left of the double commutators. Furthermore, due to Proposition \ref{prop_polynomialIdentity_even_n} (or \ref{prop_polynomialIdentity2_even_n}), we can exchange the places of variables outside the commutators. Therefore, we obtain
\begin{align*}
&h_{i_1, \dots, i_{n+1}, j} \equiv_{I_5} \sum_{\sigma \in S_{n+1}}  (-1)^{\sigma} \alpha_{\sigma} 
[x_{\sigma(1)}^{i_{\sigma(1)}},x_{\sigma(2)}^{i_{\sigma(2)}}]
\cdots [x_{\sigma(n-2)}^{i_{\sigma(n-2)}}, x_{\sigma(n-1)}^{i_{\sigma(n-1)}}] [[x_{\sigma(n-1)}^{i_{\sigma(n-1)}}, x_{\sigma(n)}^{i_{\sigma(n)}}],[x_{\sigma(n+1)}^{i_{\sigma(n+1)}}, x_1^j]]\\
&x_1^{j+1} \cdots x_1^{q_1}x_{\sigma(n+1)}^{i_{\sigma(n+1)+1}} \cdots x_{\sigma(n+1)}^{q_{\sigma(n+1)}} 
x_{\sigma(n)}^{i_{\sigma(n)+1}} \cdots x_{\sigma(n)}^{q_{\sigma(n)}}x_{\sigma(n-1)}^{i_{\sigma(n-1)+1}} \cdots x_{\sigma(n-1)}^{q_{\sigma(n-1)}},
\end{align*}
where 
\begin{align*}
  & \alpha_{\sigma} =  x_{\sigma(1)}^1\cdots \widehat{x_{\sigma(1)}^{i_{\sigma(1)}}} \cdots x_{\sigma(1)}^{q_{\sigma(1)}} \cdots x_{\sigma(n-1)}^1\cdots \widehat{x_{\sigma(n-1)}^{i_{\sigma(n-1)}}} \cdots x_{\sigma(n-1)}^{q_{\sigma(n-1)}} \times\\
&x_{\sigma(n)}^1\cdots
x_{\sigma(n)}^{i_{\sigma(n)-1}}x_{\sigma(n+1)}^1\cdots x_{\sigma(n+1)}^{i_{\sigma(n+1)-1}}x_1^1 \cdots x_1^{j-1}.
\end{align*}

The next step is to show that we can move to the left also the variables
\[
x_1^{j+1} \cdots x_1^{q_1}x_{\sigma(n+1)}^{i_{\sigma(n+1)+1}} \cdots x_{\sigma(n+1)}^{q_{\sigma(n+1)}} 
x_{\sigma(n)}^{i_{\sigma(n)+1}} \cdots x_{\sigma(n)}^{q_{\sigma(n)}}.
\]
This follows from the fact that the Jacobi identity implies $[[a,b],[c,d]]  = [a,b,c,d]+[a,b,d,c]$ for any $a,b,c,d$, and modulo $I_5$, we can move the monomials to the left of the big commutator.

Since by Proposition \ref{prop_polynomialIdentity_even_n} (or \ref{prop_polynomialIdentity2_even_n}) we can switch the places of variables outside commutators, we obtain

\begin{align} \label{eq_g_end}
&h_{i_1, \dots, i_{n+1},j}  \equiv_{I_5} \beta\sum_{\sigma \in S_{n+1}}  (-1)^{\sigma}
[x_{\sigma(1)}^{i_{\sigma(1)}},x_{\sigma(2)}^{i_{\sigma(2)}}]
\cdots [x_{\sigma(n-3)}^{i_{\sigma(n-3)}}, x_{\sigma(n-2)}^{i_{\sigma(n-2)}}] [[x_{\sigma(n-1)}^{i_{\sigma(n-1)}},x_{\sigma(n)}^{i_{\sigma(n)}}],[ x_{\sigma(n+1)}^{i_{\sigma(n+1)}}, x_1^j]],
\end{align}
where
\[
\beta = x_{1}^1\cdots \widehat{x_{1}^{i_1}} \cdots x_{1}^{q_{1}} \cdots x_{n+1}^1\cdots \widehat{x_{n+1}^{i_{n+1}}} \cdots x_{n+1}^{q_{n+1}}x_{1}^1\cdots \widehat{x_{1}^{j}} \cdots x_{1}^{q_{1}}.
\]
Notice that $h_{i_1, \dots, i_{n+1},j} $ is the polynomial $g_3^{(n+2)}$ evaluated on variables $x^{i_{\sigma(1)}}_{\sigma(1)}, \dots, x^{i_{\sigma(n+1)}}_{\sigma(n+1)}$ multiplied by $\beta$.
But $\beta$ is independent of the permutation $\sigma$, it depends only on the fixed choice of $i_1, \dots, i_{n+1}$ and $j$. In addition, $x^{i_{\sigma(1)}}_{\sigma(1)}, \dots, x^{i_{\sigma(n+1)}}_{\sigma(n+1)} \in \{x_1, \dots, x_n\}$, which implies that at least one of the variables appears twice.  This implies that for any choice of $i_1, \dots, i_{n+1}$ and $j$ 
\[
h_{i_1, \dots, i_{n+1},j}  \equiv_{I_5} 0
\]
and thus $g_3^{(n+2)}(m_1, \dots, m_{n+1}) \equiv_{I_5} 0$.
\end{proof}

\begin{prop}\label{standard}
    Let $n\geq 4$ be an even integer. Then $s_{n+2}(x_1, \dots, x_{n+2})$ is a polynomial identity for $F_n(\mathfrak{N}_4)$.
\end{prop}

\begin{proof}
    It is well known that  $s_{n+2}$ is a scalar multiple of the polynomial \[ \sum_{\sigma\in S_{n+2}}(-1)^{\sigma} p_{\sigma} = \sum_{\sigma\in S_{n+2}}(-1)^{\sigma} [x_{\sigma(1)},x_{\sigma(2)}]\cdots [x_{\sigma(n+1)},x_{\sigma(n+2)}].\]

    Since $s_{n+2}$ is alternating in $n+2$ variables, $s_{n+2}$ vanishes when evaluated in pure variables of $F_n(\mathfrak{N}_4)$. Let us show that it vanishes under evaluation by arbitrary monomials.

    We consider the evaluation of the variable $x_i$ by the monomial $m_i = x^{(i)}_1\cdots x^{(i)}_{l_i}$.

        Hence, Lemma \ref{CommutatorOfMonomials} implies that for $t, s$, the commutator $[m_t,m_s]$ will result in
        \[\sum_{i_t=1}^{l_t}\sum_{i_s=1}^{l_s}x^{(t)}_1\cdots x^{(t)}_{i_{t-1}}x^{(s)}_1\cdots x^{(s)}_{i_{s-1}}[x^{(t)}_{i_t}, x^{(s)}_{i_s}]x^{(s)}_{i_{s+1}}\cdots x^{(s)}_{l_s}x^{(t)}_{i_{t+1}}\cdots x^{(t)}_{l_t}.\]

        And $s_{n+2}$ will result in

        \begin{align*}
            & \sum_{i_1, \dots, i_{n+2}=1}^{l_1, \dots, l_{n+2}} \sum_{\sigma\in S_{n+2}}(-1)^{\sigma} x^{(\sigma(1))}_1\cdots x^{(\sigma(1))}_{i_{\sigma(1)-1}}x^{(\sigma(2))}_1\cdots x^{(\sigma(2))}_{i_{\sigma(2)-1}}[x^{(\sigma(1))}_{i_\sigma(1)}, x^{(\sigma(2))}_{i_\sigma(2)}]x^{(\sigma(2))}_{i_{\sigma(2)+1}}\cdots x^{(\sigma(2))}_{l_\sigma(2)}\\
            & x^{(\sigma(1))}_{i_{\sigma(1)+1}}\cdots x^{(\sigma(1))}_{l_\sigma(1)} \cdots x^{(\sigma(n+1))}_1\cdots x^{(\sigma(n+1))}_{i_{\sigma(n+1)-1}}x^{(\sigma(n+2))}_1\cdots x^{(\sigma(n+2))}_{i_{\sigma(n+2)-1}}[x^{(\sigma(n+1))}_{i_\sigma(n+1)}, x^{(\sigma(n+2))}_{i_\sigma(n+2)}]\\
            & x^{(\sigma(n+2))}_{i_{\sigma(n+2)+1}}\cdots x^{(\sigma(n+2))}_{l_\sigma(n+2)}x^{(\sigma(n+1))}_{i_{\sigma(n+1)+1}}\cdots x^{(\sigma(n+1))}_{l_\sigma(n+1)}.
        \end{align*}

        After such evaluation, Lemma \ref{CommutatorOfMonomials} implies that for each $i_1, \dots, i_{n+2}$ and for each $\sigma$, the corresponding polynomial is a linear combination of elements of type
        \[a_1[b_1,b_2]a_2[b_3,b_4]a_3\cdots a_{\frac{n+2}{2}}[b_{n+1},b_{n+2}]a_{\frac{n+4}{2}},\] where the $b_i$ are pure variables and the $a_i$ are monomials. From Proposition \ref{prop_polynomialIdentity2_even_n} we know that a product of $\frac{n+2}{2}$ commutators vanishes in $F_n(\mathfrak{N}_4)$ if at least one of the commutators is a triple one. As a consequence, in the above polynomial, we can move the variables outside commutators to the leftmost position. Also, since the product of any $\frac{n+4}{2}$ commutators vanishes in $F_n(\mathfrak{N}_4)$, we can reorder the variables outside the commutator in any order.

        Back to the evaluation of $s_{n+2}$ on monomials, we obtain 

         \[\sum_{i_1, \dots, i_{n+2}=1}^{l_1, \dots, l_{n+2}}q_{i_1,\dots,i_{n+2}} s_{n+2}(x^{(1)}_{i_1},\dots,x^{(n+2)}_{i_{n+2}}), \]
         for some polynomials $q_{i_1,\dots,i_{n+2}}$. 
 
        Now the result follows from the fact that $s_{n+2}$ vanishes when evaluated on $n$ pure variables, since it is alternating in $n+2$ variables.
\end{proof}

\begin{proof}(of Theorem \ref{thm_PI} for even $n \geq 4$): Let $m > n$. We proceed as in the odd case.  We need to check which of the generators of the $S_m$-module $\Gamma_m(\mathfrak{N}_4)$ given in Theorem \ref{thm_Stoyanova-Venkova} are polynomial identities for $F_n(\mathfrak{N}_4)$.

Recall that by Proposition \ref{g3n+2_even} if $2k>n$, $g_3^{(2k)}$ is a polynomial identity for $F_n(\mathfrak{N}_4)$.

Now, we consider the polynomial $g_1^{(2k-1)}$. If $2k-1 = n+1$ (i.e., $n = 2k-2$) 
then, $g_1^{(n+1)} \in K\left\langle x_1 \dots, x_n \right\rangle$. By Lemma \ref{nvariables}, $g_1^{(n+1)}$ is a polynomial identity for $F_n(\mathfrak{N}_4)$ if and only if $g_1^{(n+1)} \in I_5$. The last statement does not hold, since the linearization of $g_1^{(n+1)}$ generates a nontrivial $S_{n+1}$-submodule of $\Gamma_{n+1}(\mathfrak{N}_4)$. Therefore, $g_1^{(n+1)}$ is not a polynomial identity for $F_n(\mathfrak{N}_4)$. 

Next, let $2k-1 > n+ 1$ (i.e., $n < 2k-2$). Notice that by \cite[Theorem 3]{Stoyanova1984(N4)}, any $g_{i}^{(r)}$ is a consequence of $g_{3}^{(n+2)}$ if $r\geq n+3$. In particular, $g_1^{(2k-1)}$ is a consequence of $g_{3}^{(n+2)}$.

In the same way we show that when $n = 2k-2$, $g_2^{(2k-1)} = g_{2}^{(n+1)}$ and $g_3^{(2k-1)} = g_{3}^{(n+1)}$ are not polynomial identities for $F_n(\mathfrak{N}_4)$ and when $n < 2k-2$ then $g_2^{(2k-1)}$ and $g_3^{(2k-1)} $  are consequences of $g_{3}^{(n+2)}$.

Next, we consider the polynomials $g_1^{(2k)}$ and $g_2^{(2k)}$.

If $2k > n+2$ (i.e., $n < 2k-2$), then $g_1^{(2k)}$ and  $g_2^{(2k)}$ are consequences of $g_{3}^{(n+2)}$.  On the other hand, if $2k=n+2$ (i.e., $n=2k-2$), then 
both polynomials belong to $K\left\langle x_1, \dots, x_n \right\rangle$, hence by Lemma \ref{nvariables} they are not polynomial identities for $F_n(\mathfrak{N}_4)$. 

Finally, $s_{2k}$ is not a polynomial identity if $2k\leq n$ and is a polynomial identity if $2k>n$, by Proposition \ref{standard}. 

As a consequence, we showed that if $n\geq 6$ is even, any of the generators of the $S_m$-modules $\Gamma_m(\mathfrak{N}_4)$ are either non-identities for  $F_n(\mathfrak{N}_4)$ or consequences of the polynomials $g_3^{(n+2)}$, $s_{n+2}$ and $[x_1,x_2,x_3,x_4,x_5]$.

To finish the proof of the theorem, we need to consider the case $n=4$.
To that it is enough to prove that the polynomials $g_{0}^{(5)}$ and $g_{0}^{(6)}$ are not identities for $F_4(\mathfrak{N}_4)$.
This is immediate once both are polynomials in only 2 variables.
\end{proof}

\subsection{The cases $n=2$ and $n=3$}

The following is a consequence of Proposition \ref{standard}.

\begin{prop} \label{prop_identity_s6_n3}

Let $n=2$ or $n=3$. Then the standard polynomial 
\[
s_6(x_1, \dots, x_6) =\frac{1}{2^{3}}\sum_{\sigma \in S_{6}} (-1)^{\sigma} [x_{\sigma(1)}, x_{\sigma(2)}] [x_{\sigma(3)}, x_{\sigma(4)}] [x_{\sigma(5)}, x_{\sigma(6)}]
\]
is an identity for $F_n(\mathfrak{N}_4)$.
\end{prop} 

\begin{proof}
Since $s_6$ is an identity for $F_4(\mathfrak{N}_4)$, then it is also an identity for $F_3(\mathfrak{N}_4)$ and $F_2(\mathfrak{N}_4)$.
\end{proof}

\begin{prop}
    The standard polynomial of degree 4,
    \[s_4(x_1,x_2,x_3,x_4)=\sum_{\sigma\in S_4} (-1)^\sigma x_{\sigma(1)}x_{\sigma(2)}x_{\sigma(3)}x_{\sigma(4)}\]
    is not a polynomial identity for $F_2(\mathfrak{N}_4)$.
\end{prop}

\begin{proof}
    The proof follows once we show that the substitutions $x_3\mapsto x_1x_2x_1$ and $x_4\mapsto x_2^2$ give us a polynomial which is not in $I_5$. 

    The above evaluation yields 
    \[s_4(x_1,x_2,x_1x_2x_1,x_2^2) = x_2(-x_2x_1x_2x_1^2 -x_1^2 x_2x_1x_2 +x_1 x_2 x_1^2x_2 + x_2x_1^2 x_2x_1)x_2.\]

    We will show that $-x_2x_1x_2x_1^2 -x_1^2 x_2x_1x_2 +x_1 x_2 x_1^2x_2 + x_2x_1^2 x_2x_1\not \in I_5$ and this implies that the above is not in $I_5$ either.

    To that, notice that $-x_2x_1x_2x_1^2 -x_1^2 x_2x_1x_2 +x_1 x_2 x_1^2x_2 + x_2x_1^2 x_2x_1 = [x_1[x_2,x_1]x_1,x_2] $. Now we use Corollary \ref{Stoyanova-Venkova_corollary}. Using the evaluation $x_1 =e_1\otimes 1 \otimes 1 +1\otimes f_1 \otimes 1 + 1\otimes 1 \otimes g_1 $ and $x_2 = e_2\otimes 1 \otimes 1 +1\otimes f_2 \otimes 1 + 1\otimes 1 \otimes g_2 $ we obtain 
    \[[x_1[x_2,x_1]x_1,x_2] 
 = -16(e_1e_2\otimes f_1f_2\otimes g_1+e_1e_2\otimes f_1\otimes g_1g_2+e_1\otimes f_1f_2\otimes g_1g_2)\]
 which is nonzero in $E_2\otimes E_2\otimes E_2$. Hence $s_4$ is not a polynomial identity for $F_2(\mathfrak{N}_4)$.
\end{proof}

The analogue of Theorem \ref{thm_PI} for the case $n=2$ is stated below without a proof. The proof consists of the same arguments as in the cases $n\geq 4$, but the proofs that the polynomials $g_2^{(5)}$ and $g_3^{(5)}$ are identities are much longer.

\begin{theorem}
    The polynomial identities of $F_2(\mathfrak{N}_4)$ are given by
    \[\mathrm{Id}(F_2(\mathfrak{N}_4)) = \langle g_2^{(5)}, g_3^{(5)}, s_6, [x_1,x_2,x_3,x_4,x_5]\rangle^T\]
    where
    \[g_2^{(5)} = \sum_{\sigma\in S_3}(-1)^{\sigma} [x_{\sigma(1)},x_{\sigma(2)}][x_2,x_1,x_{\sigma(3)}],\]
    \[g_3^{(5)} = \sum_{\sigma\in S_4}(-1)^{\sigma} [x_{\sigma(1)},x_{\sigma(2)}][x_{\sigma(3)}, x_{\sigma(4)},x_1].\]
\end{theorem}

For the case $n=3$ we have some partial results.

\begin{lemma}
    The polynomials $g_1^{(6)}$, $g_2^{(6)}$ and $g_3^{(6)}$ are polynomial identities for $F_3(\mathfrak{N}_4)$.
\end{lemma}

\begin{proof}
    Notice that by Proposition \ref{standard4} $s_4$ is an identity for $F_3(\mathfrak{N}_3)$. So the three polynomials above belong to $I_2I_4$ and by Theorem \ref{thm_products_of_commutators} (vi), they lie in $I_5$. 
\end{proof}

Notice that $g_2^{(5)}$ is not an identity for $F_3(\mathfrak{N}_4)$ by Lemma \ref{nvariables}. Since $g_2^{(5)}$ is a consequence of $s_4$, it follows that $s_4$ is not an identity for $F_3(\mathfrak{N}_4)$ either. 
By the above, we have only two possibilities for $\mathrm{Id}(F_3(\mathfrak{N}_4))$. Either it is generated as a T-ideal by $g_3^{(5)}$, $s_6$, and $[x_1,x_2,x_3,x_4,x_5]$ or by $g_1^{(6)}$, $g_{2}^{(6)}$, $g_3^{(6)}$, $s_6$, and $[x_1,x_2,x_3,x_4,x_5]$.

It will be the first case if $g_3^{(5)}$ is an identity for $F_3(\mathfrak{N}_4)$ and the second case if $g_3^{(5)}$ is not an identity for $F_3(\mathfrak{N}_4)$. Therefore, to complete the case $n=3$ one needs to show whether $g_3^{(5)}$ is or is not an identity for $F_3(\mathfrak{N}_4)$.

\section{The variety $\mathfrak{N}_p$ for arbitrary $p$}

In this section, we consider the variety of Lie nilpotent associative algebras defined by the identity $[x_1, \dots, x_{p+1}] = 0$, where $p$ is an arbitrary positive integer. Since $\mathfrak{N}_p$ is a non-matrix variety, and $F_n(\mathfrak{N}_p)$ is finitely generated, it follows that $F_n(\mathfrak{N}_p)$ satisfies an identity of the type $[x_1, x_2]\cdots [x_{2k-1}, x_{2k}]$ for some integer $k$ \cite{Cekanu1980, MischenkoPetrogradskyRegev2011}. In this section, for all $p$ and $n$ we find a lower bound for $k$. 
In other words, we find $k$ such that $[x_1, x_2]\cdots [x_{2k-1}, x_{2k}]$ is not an identity for $F_n(\mathfrak{N}_p)$.

\begin{prop}
Consider the relatively free algebra $F_n(\mathfrak{N_p})$ for arbitrary $p$ and $n$. Then the following hold:
\begin{itemize}
\item[(i)] If one of $n$ and $p$ is odd and the other one is even then
\[
g = [x_1, x_2]\cdots [x_{n+p-4}, x_{n+p-3}]
\] 
is not a polynomial identity for $F_n(\mathfrak{N}_p)$.

\item[(ii)] If both $n$ and $p$ are odd then
\[
g = [x_1, x_2]\cdots [x_{n+p-5}, x_{n+p-4}]
\]
is not an identity for $F_n(\mathfrak{N}_p)$.
\item[(iii)] If both $n$ and $p$ are even and then
\[
g = [x_1, x_2]\cdots [x_{n+p-3}, x_{n+p-2}]
\]
is not an identity for $F_n(\mathfrak{N}_p)$.
\end{itemize}
\end{prop}

\begin{proof}
We first consider the case where both $n$ and $p$ are even. We will show that
\[
g= [x_1, x_2]\cdots [x_{n+p-3}, x_{n+p-2}]
\]
is not an identity for $F_n(\mathfrak{N}_p)$. The polynomial
\[
g' = [x_1, x_2][x_1, x_3]\cdots[x_1, x_p][x_{p+1}, x_{p+2}]\cdots[x_{n-1}, x_n]
\]
is a consequence of $g$ and is a polynomial in $n$ variables. Hence, by Proposition \ref{prop_DeryabinaKrasilnikov}, it is enough to show that $g'$ is not an identity for $E\otimes E_{p-2}$. Here we recall that $E$ denotes the Grassmann algebra over a countable dimensional vector space with basis $\{e_1, e_2, \dots, e_n, \dots \}$ and $E_{p-2}$ denotes the Grassmann algebra over a $(p-2)$-dimensional vector space with basis $\{g_1, g_2, \dots, g_{p-2}\}$.

We consider the evaluation
\begin{align*}
&x_1 \mapsto e_1\otimes g_1 + e_2\otimes g_2 + \dots + e_{p-2}\otimes g_{p-2} + e_{p-1} \otimes 1 \\
&x_i \mapsto e_{p-2 + i} \otimes 1 \text{ for } i\geq 2.
\end{align*}

By direct computation we show that 
\[
g'(e_1\otimes g_1 + e_2\otimes g_2 + \dots + e_{p-2}\otimes g_{p-2} + e_{p-1} \otimes 1, e_{p}\otimes 1, \dots, e_{p+n-2}\otimes 1) \neq 0.
\]
Hence, $g'$ is not a polynomial identity for $E\otimes E_{p-2}$ and therefore $g$ is not an identity for $F_n(\mathfrak{N}_p)$ for even $p$ and $n$. This proves Part (iii) of the proposition. In addition, $g$ not being an identity for $F_n(\mathfrak{N}_p)$ implies that $g$ is not an identity for $F_{n+1}(\mathfrak{N}_p)$ either. This proves Part (i) of the proposition for odd $n$ and even $p$.
Furthermore, $g$ not being an identity for $F_n(\mathfrak{N}_p)$ also implies that $g$ is not an identity for $F_n(\mathfrak{N}_{p+1})$. This proves Part (i) of the proposition for even $n$ and odd $p$. 
Finally, $g$ not being an identity for $F_n(\mathfrak{N}_{p+1})$ implies that $g$ is not an identity for $F_{n+1}(\mathfrak{N}_{p+1})$ either. This proves Part (ii) of the proposition.
\end{proof}
An upper bound for $k$ is given in the following

\begin{prop}
    If $k\geq np-n+1$ is even, then $F_n(\mathfrak{N}_p)$ satisfies the identity $f=[x_1,x_2]\cdots [x_{k-1},x_k]$.
\end{prop}

\begin{proof}
    Let $m_1, \dots, m_k$ be monomials in the variables $x_1, \dots, x_n$. Then $f(m_1, \dots, m_k)$ will be a linear combination of elements of the type
    \[a_1[b_1,b_2]a_2[b_3,b_4]a_3\cdots [b_{k-1},b_k]a_r\] where $a_i$ are monomials in $x_1, \dots x_n$ and $b_i\in \{x_1, \dots, x_n\}$ are pure variables.
    By the pigeonhole principle, if we have $k\geq np-n+1 = n(p-1)+1$ then for each summand above, at least one of the $n$ variables $x_1, \dots, x_n$ will appear $p$ times in $b_1, \dots, b_k$. If $p$ is even, then $f\in I_3^{\frac{p}{2}}\subseteq I_{p+1}$. If $p$ is odd, then $f\in I_3^{\frac{p-1}{2}}I_2\subseteq I_{p}I_2\subseteq I_{p+1}$.
\end{proof}

\begin{question}
    What is the minimal $k$ such that $F_n(\mathfrak{N}_p)$ satisfies the identity $[x_1,x_2]\cdots [x_{k-1},x_k]$?
\end{question}

We conjecture that the minimal $k$ is the following:
\begin{itemize}
    \item $k=n+p-1$ if one of $n$ and $p$ is even and the other is odd and $n \geq p+1$.
    \item $k = n+p-2$ if both $n$ and $p$ are odd and $n\geq p+2$.
    \item $k=n+p$ if both $n$ and $p$ are even and $n\geq p$.
\end{itemize}

By a theorem of Kemer \cite{K1991}, for every finitely generated associative PI-algebra $A$ over an infinite field $K$ there exists $d \geq 2$ such that $A$ satisfies the standard identity $s_d(x_1, \dots, x_d)$ of degree $d$. 
So a natural question here is the following:

\begin{question}
    Find the minimal $k$ such that $F_n(\mathfrak{N}_p)$ satisfies $s_k(x_1,\dots,x_k) = 0$.
\end{question}

We conjecture that 
\begin{itemize}
	\item $k = n+p-1$, if one of $n$ and $p$ is even and the other is odd and $n \geq p+1$.
	\item $k= n+p-2$, if both $n$ and $p$ are odd and $n \geq p+2$.
	\item $k= n+p - 2$, if both $n$ and $p$ are even and $n \geq p$.
\end{itemize}

These conjectures reflect our results in the cases $p=3$ and $p=4$.

\section{Asymptotic equivalence of varieties}

Classifying the subvarieties of a given variety can be a rather intricate task. A coarser approach to comparing varieties (or their corresponding T-ideals) is provided by the notion of asymptotic equivalence. We say that the varieties $\mathfrak{N}$ and $\mathfrak{M}$ are asymptotically equivalent if they satisfy the same proper polynomial identities of large enough degree. This notion was introduced by Kemer in \cite{Kemer1990}. Here we want to compare the varieties of Lie nilpotent associative algebras of different indexes using this notion of asymptotic equivalence.

\begin{prop}
    The varieties $\mathfrak{N}_2$ and $\mathfrak{N_3}$ are asymptotically equivalent.
\end{prop}

\begin{proof}
    Let us prove that if $n\geq 5$ and $f\in I_3$ is a proper multilinear polynomial of degree $n$, then $f\in I_4$. To that end, let us denote the polynomial $f$ modulo $I_4$ by $\bar{f}$. Then we have $\bar{f}\in \Gamma_n(\mathfrak{N}_3)$. By Proposition \ref{Volichenko} and since $\deg f\geq 5$, either $\bar{f} = 0$ or $\bar{f}$ is a scalar multiple of $u_{n,n}$ modulo $I_4$. But $u_{n,n}$ is the standard polynomial of degree $n$, and we know that it is not an element of $I_3$. This means that $\bar{f} = 0$ and hence $f\in I_4$.
\end{proof}

\begin{prop}
    Let $k\geq 1$. Then $\mathfrak{N}_{2k-1}$ and $\mathfrak{N}_{2k}$ are not asymptotically equivalent.
\end{prop}

\begin{proof}
    Notice that $I_{2k} = \mathrm{Id}(\mathfrak{N}_{2k-1})$ and $I_{2k+1} = \mathrm{Id}(\mathfrak{N}_{2k})$. Of course $I_{2k+1}\subseteq I_{2k}$. We will show that there are proper multilinear polynomials of arbitrarily large degree that are in $I_{2k}$ but are not in $I_{2k+1}$.

    Let $u=[x_1,\dots,x_{2k}]$ and define the polynomial 
    \[f_m = u [x_{2k+1},x_{2k+2}]\cdots [x_{m-1},x_m].\]
    Since $u\in I_{2k}$, we have $f_m\in I_{2k}$. We will show that for any $m$, $f_m\not \in I_{2k+1}$.
    By Proposition \ref{prop_DeryabinaKrasilnikov}, $I_{2k+1}\subseteq \mathrm{Id}(E\otimes E_{2k-2})$.
    So it is enough to show that $f_m$ is not an identity for $E\otimes E_{2k-2}$.
    We consider the following evaluation:
    
    $x_i\mapsto e_i\otimes 1$ if $i=1$ or $i\geq 2k$.

    $x_i\mapsto e_i\otimes g_{i-1}$ for $i=2, \dots, 2k-1$.

    Under the above evaluation, we have 

    $u = 2^{2k-1}e_1\cdots e_{2k} \otimes g_1\cdots g_{2k-2} $

    and since $[e_{2k+1}\otimes 1,e_{2k+2}\otimes 1]\cdots [e_{m-1}\otimes 1,e_{m}\otimes 1] = 2^{\frac{m}{2}-k}e_{2k+1}\cdots e_m\otimes 1$, we obtain

    $f_m = 2^{k-1+\frac{m}{2}}e_1\cdots e_{m}\otimes g_1\cdots g_{2k-2}$, which is nonzero in $E\otimes E_{2k-2}$.
\end{proof}

It is interesting to know if an equivalence similar to $\mathfrak{N}_2$ and $\mathfrak{N}_3$ holds for higher values of $n$.

\begin{question}
    For $k\geq 2$, are $\mathfrak{N}_{2k}$ and $\mathfrak{N}_{2k+1}$ asymptotically equivalent?
\end{question}

\section{Funding}
The research of E.H. was supported, in part, by the Bulgarian Ministry of Education and Science, Scientific Programme ”Enhancing the Research Capacity in Mathematical Sciences (PIKOM)”, No. DO1-241/15.08.2023.
T.C.M. was supported, in part, by the São Paulo Research Foundation (FAPESP), Brazil, process number 2018/23690-6 and by CNPq -- Brazil, process number 405779/2023-2.

\end{document}